\documentclass{amsart}
\usepackage{amsmath,amsfonts,amssymb,amscd,amsthm,epsf}
\usepackage{pinlabel}
\newtheorem{prop}{Proposition}[section]
\newtheorem{thm}[prop]{Theorem}
\newtheorem{cor}[prop]{Corollary}

\newtheorem{lem}[prop]{Lemma}

\theoremstyle{definition}
\newtheorem{de}[prop]{Definition}

\theoremstyle{remark}
\newtheorem*{remarks}{Remarks}            

\def\CP{{\mathbb C \mathbb P}}
\def\Z{{\mathbb Z}}
\def\R{{\mathbb R}}

\def\rot{{\mathcal R}}
\def\emb{{\hookrightarrow}}
\def\inter{\mathop{\rm int}\nolimits}

\def\id{\mathop{\rm id}\nolimits}

\def\SL{\mathop{\rm SL}\nolimits}

\def\SO{\mathop{\rm SO}\nolimits}

\def\co{\colon\thinspace}

\def\Diff{\mathop{\rm Diff}\nolimits}
\begin{document}
\title{Infinite order corks via handle diagrams}
\author{Robert E. Gompf}
\address{The University of Texas at Austin, Mathematics Department RLM 8.100, Attn: Robert Gompf,
2515 Speedway Stop C1200, Austin, Texas 78712-1202}
\email{gompf@math.utexas.edu}
\begin{abstract} The author recently proved the existence of an infinite order cork: a compact, contractible submanifold $C$ of a 4-manifold and an infinite order diffeomorphism $f$ of $\partial C$ such that cutting out $C$ and regluing it by distinct powers of $f$ yields pairwise nondiffeomorphic manifolds. The present paper exhibits the first handle diagrams of this phenomenon, by translating the earlier proof into this language (for each of the infinitely many corks arising in the first paper). The cork twists in these papers are twists on incompressible tori. We give conditions guaranteeing that such twists do not change the diffeomorphism type of a 4-manifold, partially answering a question from the original paper. We also show that the ``$\delta$-moves" recently introduced by Akbulut are essentially equivalent to torus twists.
\end{abstract}
\maketitle


\section{Introduction}

The failure of high-dimensional topology to apply to smooth 4-manifolds led to the notion of a {\em cork twist}. As originally formulated, this consists of changing the diffeomorphism type of a closed 4-manifold $X$ by removing a compact, contractible, smooth submanifold $C$ from $X$ and regluing it by an involution $f$ of $\partial C$. The first example of a cork twist was published by Akbulut \cite{A} in 1991. A few years later, various authors \cite{CFHS}, \cite{M} showed that any two homeomorphic, simply connected (smooth) 4-manifolds are related by a cork twist. (See \cite{GS}, \cite{InfCork} for more history). The question was immediately raised of whether higher order corks may exist---and in particular, whether there was such a pair $C\subset X$ and an infinite order diffeomorphism $f$ of $\partial C$ such that the $\Z$-indexed family of homeomorphic manifolds obtained by regluing using all powers of $f$ were pairwise nondiffeomorphic. In weaker form, can there even be a contractible 4-manifold $C$ with a boundary diffeomorphism for which no nonzero power extends to a self-diffeomorphism of $C$? No progress was made on these questions until recently. Corks of all finite orders were constructed in 2016 by Tange~\cite{T} and Auckly, Kim, Melvin and Ruberman~\cite{AKMR}. A withdrawn 2014 posting of Akbulut \cite{withdrawn} attempted to construct infinite order corks of the weaker sort, using handle calculus. In 2016 \cite{InfCork}, the author of the present paper constructed an infinite family of examples $C(r,s;m)\subset X$ (for $r,s>0>m$), each one affirmatively answering the stronger question, using an entirely different plan of attack and no handle diagrams. This raised the question of how the paper translates into the language of these diagrams. Section~\ref{Notes} of the present paper gives the result, in a form independent of, but clearly derived from, \cite{InfCork}. This section shows boundary diffeomorphisms of contractible 4-manifolds, and how they can provide infinitely many diffeomorphism types of ambient 4-manifolds, as presented in the currently preferred language of the subject. We also find conditions under which such diffeomorphisms necessarily do extend over a given 4-manifold, partially answering a question from \cite{InfCork}. A recent paper of Akbulut~\cite{Acork} seeks to generalize the twists of Section~\ref{Notes}; our final section shows that his viewpoint is equivalent to ours.

After a quick exposition of the relevant 3-manifold diffeomorphisms in Section~\ref{TorusTwists}, this paper proceeds with three independent sections, beginning with our translation of the proofs in \cite{InfCork} into handle calculus (Section~\ref{Notes}). We sketch the correspondence between the proofs as we go along. We exhibit a cork $C$ by a diagram (Figure~\ref{C}) constructed from the existence proof of \cite{InfCork}, and show it is diffeomorphic to $C(r,s;m)$ as exhibited by Figure~\ref{Crsm}. Then we embed $C$ in a family of larger manifolds $Z_k(r,s;m)$, related (as $k$ varies in $\Z$, with the other variables fixed) by powers of a twist parallel to an incompressible torus in $\partial C$ (Figure~\ref{Z}). Finally, we show that these manifolds embed in a family of closed manifolds $X_k$ related by the same cork twists, and distinguish these using the same method as \cite{InfCork}: We show that they are obtained by the Fintushel-Stern knot construction on elliptic surfaces.

Section~\ref{Fishtail} gives our criterion guaranteeing that cutting and regluing does not change the diffeomorphism type of a 4-manifold. We partially answer a question from \cite{InfCork}: The nontrivial cork twists of that paper were diffeomorphisms of $\partial C(r,s;m)$ twisting along an incompressible torus parallel to its longitude. It was asked whether twisting parallel to the meridian was also nontrivial. We show that the answer is no for a family of manifolds including each $C(r,s;-1)$. Thus, while the potential torus twists are indexed by $H_1(T^2)\cong\Z\oplus \Z$, only one $\Z$-summand is useful for producing exotic 4-manifolds. The question remains open for $m<-1$, but we also see that for the specific embedding $C(r,s;m)\subset X$ used in \cite{InfCork}, only one $\Z$-summand (the longitude) affects the resulting diffeomorphism type. Thus, a nontriviality proof for meridian twisting would at least require a different setup. Coupling our meridian twist criterion with a recent observation of Ray and Ruberman, we see that there are contractible manifolds $C$ for which every boundary diffeomorphism extends over $C$ even though $\partial C$ contains an incompressible torus.

Our final section concerns the recent paper \cite{Acork}. Starting from the preliminary version of Section~\ref{Notes} of the present paper, Akbulut sketched the proof of the simplest case $C(1,1;-1)$. The apparent motivation was to introduce the notion of {\em $\delta$-moves} as an alternative to torus twists. Such moves depend on a choice of auxiliary band (and other data not specified in that paper), so appear to provide additional generality. We show, however, that $\delta$-moves are essentially equivalent to torus twists under broad hypotheses, for example, on irreducible homology spheres (Corollary~\ref{Hsphere}). Since incompressible tori are tightly constrained in 3-manifolds, $\delta$-moves are harder to find than they initially seem to be, and might be most easily located using the well-developed theory of incompressible tori (as Akbulut implicitly did in obtaining the proof he posted from our Section~\ref{Notes}). To obtain our equivalence, we must first address foundational issues such as well-definedness of $\delta$-moves. We also observe technical difficulties that need to be addressed whenever $\delta$-moves (or torus twists) are used for diagrammatically cutting and pasting 4-manifolds. 

\begin{remarks} (a) It is known that a cork twist cannot change the homeomorphism type of a 4-manifold, since every boundary diffeomorphism $f$ of a contractible 4-manifold $C$ extends over it homeomorphically. For a short proof, use $f$ to glue two copies of $C$ along their boundary, obtaining a homotopy 4-sphere that automatically bounds a contractible 5-manifold $W$ (via a smooth h-cobordism to $S^4$, or by working topologically and observing that $\partial W$ is homeomorphic to $S^4$ by Freedman \cite{F}). We can view $W$ as a topological h-cobordism of $C$ with a fixed product structure over $\partial C$ realizing $f$. Freedman's h-Cobordism Theorem \cite{F} extends the product structure, and projecting to $C$ extends $f$.

(b) Some authors require corks to be Stein domains by definition. This seems to be an entirely separate issue from that of changing diffeomorphism types by twisting, although Akbulut and Matveyev \cite{AM} showed that corks, in the original sense where $f$ is an involution, can always be modified to admit Stein structures. The author has made no attempt to address the Stein condition in this paper or its predecessor. It remains an interesting question whether any of these corks are (or can be modified to be) Stein domains.
\end{remarks}

We work in the smooth category throughout the paper. For simplicity, we assume (unless otherwise indicated) that all 3-manifolds are orientable and closed, and all 4-manifolds are orientable and compact (allowing boundary).


\section{Torus twists}\label{TorusTwists}

We begin with a quick exposition of the 3-manifold diffeomorphisms that will be central to this paper. Let $T\subset M$ be a torus embedded in a 3-manifold.  Identify a tubular neighborhood of $T$ with $S^1\times S^1\times I$, and let $\alpha$ and $\beta$ denote $S^1\times\{\zeta\}\times\{0\}$ and $\{\zeta\}\times S^1\times\{0\}$, respectively, for some $\zeta\in S^1$.

\begin{de}\label{Ttwist} The {\em torus twist} on $T$ parallel to $\alpha$ is the diffeomorphism $f\co M\to M$ obtained from $f(\theta,\phi,t)=(\theta+2\pi t, \phi,t)$ by extending as the identity on the rest of $M$ and smoothing.
\end{de}

\noindent Informally, we cross a Dehn twist on the annulus $\alpha\times I$ with the identity on $\beta$. This is a well-known, classical diffeomorphism of $M$, sometimes called a {\em Dehn twist} on $T$. Since $T$ can be identified with $S^1\times S^1$ so that $\alpha$ represents any preassigned primitive homology class of $T$, every element of $H_1(T)$ determines some power of a torus twist. More formally, $T$ is contained, by Lie group multiplication, in the group $\Diff_+(T)$ of orientation-preserving self-diffeomorphisms of $T$, inducing an isomorphism $\pi_1(T)\to\pi_1(\Diff_+(T))$. This descends by torus twisting to a homomorphism $\Z\oplus\Z\cong\pi_1(T)\to\pi_0(\Diff_+(M))$. For example, when $M$ is a torus bundle, all self-diffeomorphisms fixing one fiber pointwise arise by twisting on another fiber. For most irreducible 3-manifolds $M$, torus twists on all incompressible tori together generate the group $\pi_0(\Diff_+(M))$ up to a finite extension. (See the last corollary of Waldhausen \cite{W}.) When $T$ is compressible and $M$ is irreducible, $T$ either lies in a ball (bounded by the compressed torus) or bounds a solid torus over which the action of $T$ extends, so all twists on $T$ are isotopic to the identity. Note that irreducibility is necessary: For a connected sum, we expect nontrivial slide diffeomorphisms constructed by dragging the site of the sum around a loop $\gamma$ in one summand. Such a diffeomorphism can be described as a twist about the compressible torus bounding a tubular neighborhood of $\gamma$.

We can similarly define twists on Klein bottles. We will see in Section~\ref{Twist} that these are less useful than torus twists, but we introduce them for completing the discussion there of $\delta$-moves. If $K\subset M$ is an embedded Klein bottle, we identify $K$ as a bundle over a circle $\beta$ with fiber $\alpha$. The previous description can still be applied over intervals in $\beta$, since the monodromy around $\beta$ reverses orientations of both $\alpha$ and $I$ (by orientability of $M$), hence, commutes with the Dehn twist. 

Both kinds of twists have a convenient surgery description. First, draw a framed link diagram of $M$ so that $\alpha$ is an unknot in the ambient $S^3$ with the torus or Klein bottle $T$ inducing the 0-framing. Such a diagram can be obtained from an arbitrary diagram, in which $\alpha$ appears as a framed knot, by blowing up to change some crossings and adjust its framing. (One can also simultaneously control $\beta$ so that the two curves bound disks in the ambient $S^3$ with disjoint interiors.) To realize the twist, blow up a $\pm 1$-framed curve $\gamma$ at $\alpha$, slide it around $T$ as in Figure~\ref{alpha} until it returns to its original position, and blow it back down. Any curve intersecting $T$ must slide over $\gamma$ as it passes, creating the required twist. To verify that this gives a well-defined diffeomorphism, it is not necessary to see that $T$ is embedded, only that $\gamma$ returns to its original position after sliding around $M$ (to allow conjugation by the diffeomorphism from $M$ to its blowup.)

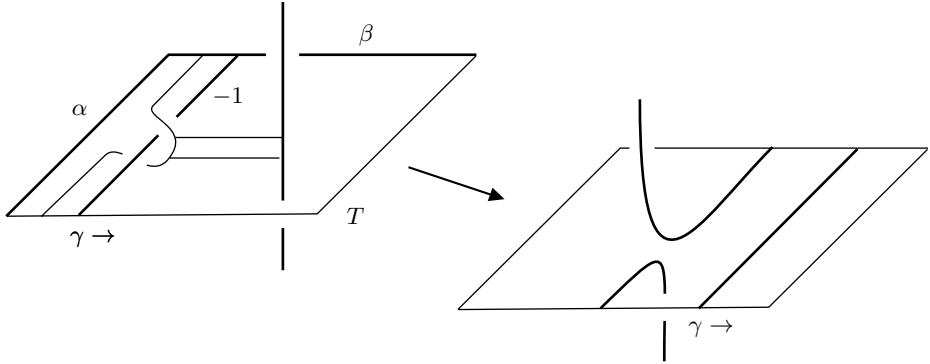
\begin{figure}
\labellist
\small\hair 2pt
\pinlabel $\alpha$ at 28 95
\pinlabel $\beta$ at 136 123
\pinlabel $\gamma\to$ at 33 46
\pinlabel $\gamma\to$ at 266 12
\pinlabel $-1$ at 84 100
\pinlabel $\gamma\to$ at 33 46
\pinlabel $T$ at 132 55
\endlabellist
\centering
\includegraphics{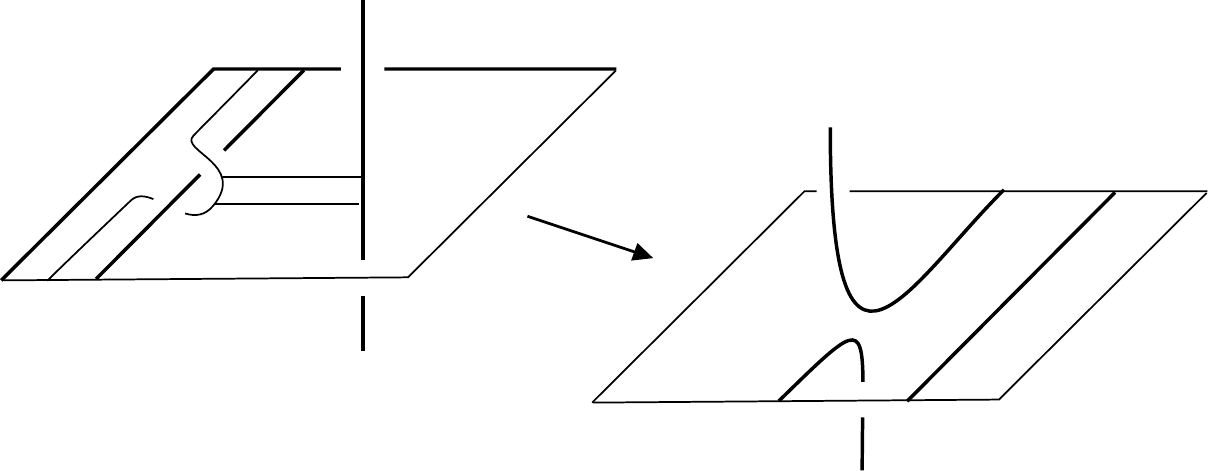}
\caption{Torus twisting.}
\label{alpha}
\end{figure}

\section{Diagrams of Corks}\label{Notes}

\begin{figure}
\labellist
\small\hair 2pt
\pinlabel $-s$ at 19 41
\pinlabel $r$ at 85 41
\endlabellist
\centering
\includegraphics{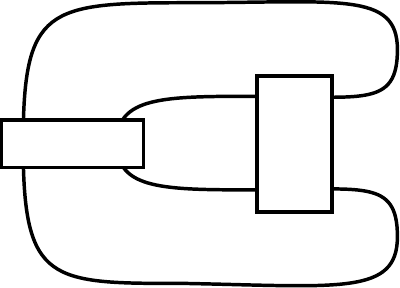}
\caption{The double twist knot $\kappa(r,-s)$.}
\label{knot}
\end{figure}

The first examples of infinite order corks were constructed in \cite{InfCork}. Let $E(r,s)$ denote the complement of the double twist knot $\kappa(r,-s)$ shown in Figure~\ref{knot} (where the boxes count full twists). The manifold $C(r,s;m)$ obtained from $I\times E(r,s)$ by adding a 2-handle along an $m$-framed meridian in $\{1\}\times E(r,s)$ is contractible. Its boundary has an obvious incompressible torus $T$, namely $\{0\}\times\partial E(r,s)$. (In fact there is a pair of incompressible tori, exhibited by moving the 2-handle to the middle level $\frac12$ and taking the tori at levels 0, 1. These are interchanged by an orientation-reversing symmetry of the construction, and they are only parallel when $|m|=1$.) Let $f\co\partial C(r,s;m)\to \partial C(r,s;m)$ be the torus twist on $T$ parallel to the longitude $\lambda$ of $T$ (the curve bounding a Seifert surface in $\{0\}\times E(r,s)$). In \cite{InfCork} it was shown that for fixed $r,s,n>0>m$, there is a canonical embedding $C(r,s;m)\emb X'_0=E(n)\# (r+s-m-3)\overline{\CP^2}$ into a blown up elliptic surface. It was also shown that the manifold $X'_k$ obtained from $X'_0$ by cutting out $C(r,s;m)$ and regluing it using $f^k$ is the corresponding blowup of the manifold $X_k$ obtained from $E(n)$ by the Fintushel-Stern construction on the knot $\kappa(k,-1)$. (The blowups can usually be avoided.) Since these latter manifolds for $k\in\Z$ are known to be pairwise nondiffeomorphic \cite{FS}, \cite{FS2}, each choice of $r,s>0>m$ yields an infinite order cork. In particular, the self-diffeomorphisms $f^k$ are not related to each other by any self-diffeomorphism of the cork, or to the identity unless $k=0$. (This last sentence applies to a larger range of $r,s,m$, due to the obvious symmetries $C(-r,-s;m)=C(r,s;m)=C(s,r;m)$ and the reflection reversing the sign of $m$. However, it is crucial that $rs>0$, ie, the twists in Figure~\ref{knot} have opposite handedness. Otherwise, $C(1,-1;-1)$ is a counterexample; see Corollary~\ref{noncork} and its preceding discussion.) The existence proof of an infinite order cork producing the family $\{ X_k\}$ did not use any handle diagrams, and recognizing the corks required only some simple 3-dimensional surgery. Since it seems useful to understand the proofs using diagrams, we now provide their translations. The resulting proofs are independent of \cite{InfCork} and almost entirely handle-theoretic, but seem unlikely to have been conceived without benefit of the abstract version.

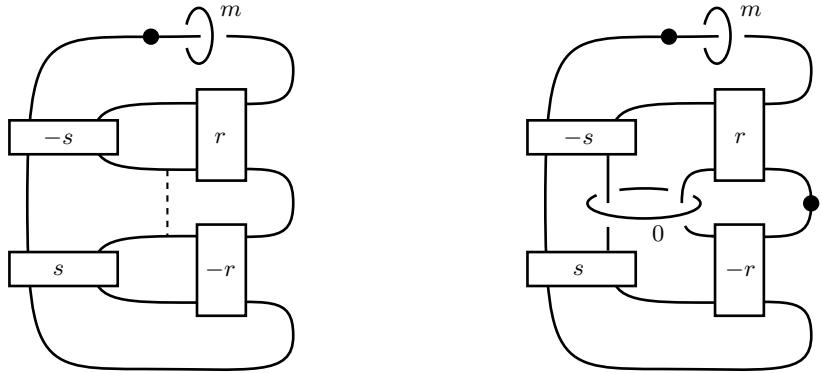
\begin{figure}
\labellist
\small\hair 2pt
\pinlabel $-s$ at 19 88
\pinlabel $s$ at 19 38
\pinlabel $r$ at 80 88
\pinlabel $-r$ at 80 38
\pinlabel $m$ at 84 136
\pinlabel $-s$ at 215 88
\pinlabel $s$ at 215 38
\pinlabel $r$ at 276 88
\pinlabel $-r$ at 276 38
\pinlabel $m$ at 280 136
\pinlabel $0$ at 245 52
\endlabellist
\centering
\includegraphics{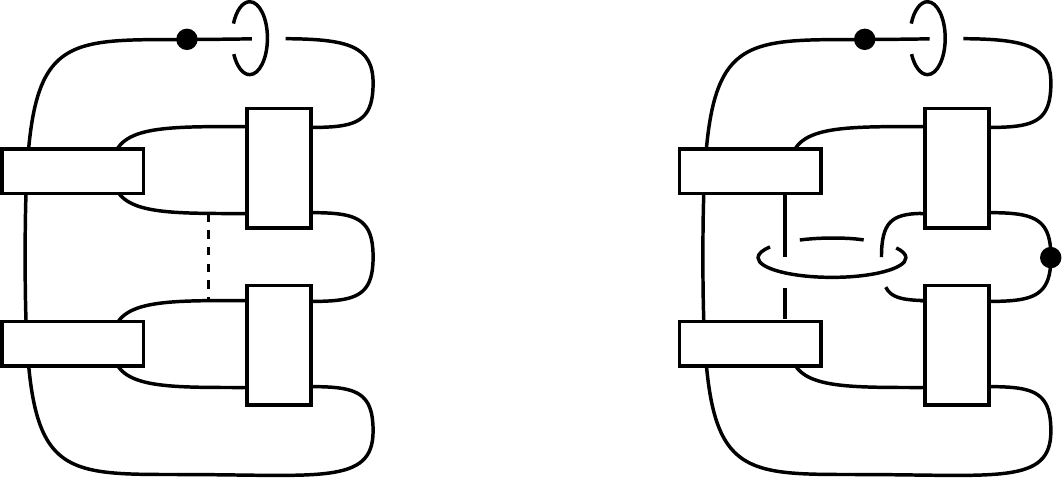}
\caption{The cork $C(r,s;m)$.}
\label{Crsm}
\end{figure}

To draw diagrams of the manifolds in \cite{InfCork}, we need diagrams of products of various 3-manifolds with $I$ and $S^1$. We use a method that was pioneered by Akbulut and Kirby, eg \cite{AK}, \cite{AkScharl}. A detailed exposition is given in \cite[Section 6.2]{GS}, particularly the solved Exercise~6.2.5(b) (and Example~4.6.8 for products with $S^1$). We illustrate with the diagrams of $C(r,s;m)$ in Figure~\ref{Crsm}. Ignoring the $m$-framed meridian in each diagram, we obtain $I\times E(r,s)$. To understand the resulting diagram on the left, consider the horizontal  plane of reflection. Adding its point at infinity and then thickening, we obtain an embedding of $I\times S^2$ in $S^3$ that represents the lateral boundary of $I\times B^3$. The upper and lower complementary 3-balls in the diagram represent $\{1\}\times B^3$ and $\{0\}\times B^3$, respectively. The dotted knot $\kappa(r,-s)\#-\kappa(r,-s)$ in the diagram is the boundary of the ribbon disk $I\times \kappa(r,-s)\subset I\times B^3$ (where we interpret the knot as a tangle in $B^3$). This disk, which can also be interpreted as the half-spin of the tangle in $B^4$, should be deleted from $B^4$ to obtain $I\times E(r,s)$ (as the dot indicates). The dashed arc represents the ribbon move transforming the dotted knot into a 2-component unlink, exhibiting the ribbon disk with a pair of local minima and a saddle point. If we use this decomposition to build a handle diagram of the complement, we obtain the diagram on the right, with the dotted unlink representing the pair of local minima and the 0-framed 2-handle arising when the saddle is submerged into the 4-ball. The $m$-framed meridian in each diagram is the 2-handle specified in the definition of $C(r,s;m)$. To see the torus $T$ in the left diagram, note that the bisecting 2-sphere in $S^3$ intersects the knot in two points. Remove these intersections by ambient surgery, using a tube comprising the boundary of a tubular neighborhood of the lower knot $-\kappa(r,-s)=\kappa(-r,s)$. The resulting torus in $\partial C(r,s;m)$ is $T=\{0\}\times\partial E(r,s)$ (seen in the boundary orientation inherited from $C(r,s;m)$). This torus is also visible in the right picture, running twice over the 0-framed 2-handle (corresponding to the two intersections of $T$ with the dashed arc in the left picture).

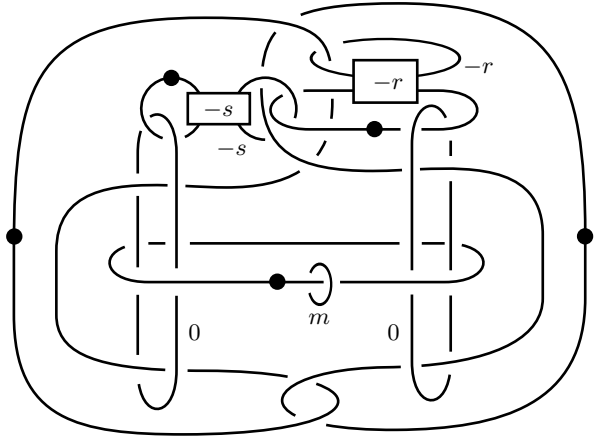
\begin{figure}
\labellist
\small\hair 2pt
\pinlabel $-s$ at 85 108
\pinlabel $-s$ at 80 123
\pinlabel $-r$ at 144 133
\pinlabel $-r$ at 178 139
\pinlabel $m$ at 118 44
\pinlabel $0$ at 71 39
\pinlabel $0$ at 146 39
\endlabellist
\centering
\includegraphics{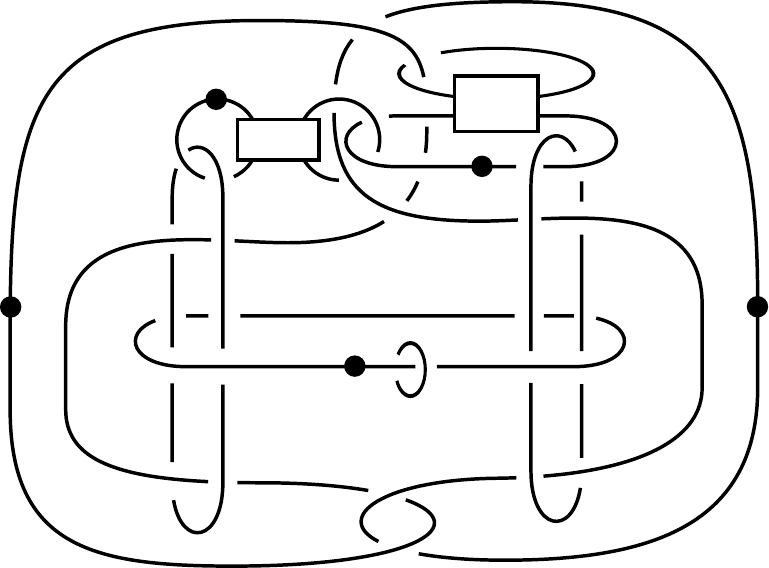}
\caption{The cork $C\approx C(r,s;m)$.}
\label{C}
\end{figure}

To exhibit infinite order corks, we need a very different description of $C(r,s;m)$ and its boundary diffeomorphism. The 4-manifold $C$ shown in Figure~\ref{C} arises from the proof of the main theorem of \cite{InfCork}. We discuss how this figure arises and then show that $C$ is diffeomorphic to $C(r,s;m)$. (This entire discussion could be excised to leave a complete but mysterious proof that $C$ is an infinite order cork. Note that it is obviously contractible, being simply connected with Euler characteristic 1.) In \cite{InfCork}, the cork $C$ was constructed from $Y=I\times\Sigma\times S^1$, where $\Sigma$ is a punctured torus, by adding three 2-handles and then drilling out the cores of two of them (connected by annuli to the far boundary of the product with $I$). If we modify Figure~\ref{C} by removing all four circles passing through twist boxes, as well as the $m$-framed meridian, what remains is this $Y$. We can see this by unwinding the two large dotted circles from each other, but it is more instructive to view the picture as a product with $S^1$: Starting from a trivial proper embedding $\Sigma\subset B^3$ with $\lambda=\partial \Sigma$ on the equator of $\partial B^3$, we obtain its tubular neighborhood $I\times\Sigma$ as the complement of a clasped pair of arcs, with the boundary of each in a single hemisphere. By the method used in the previous paragraph, this picture becomes the clasp of the two large dotted circles in the top center of the figure, and its mirror image at the bottom. Thus, we have $I\times\Sigma\times I$ exhibited as a 0-handle and two 1-handles. The algorithm for changing a product with $I$ to a product with $S^1$ introduces a $(k+1)$-handle for each $k$-handle of the original diagram. In this case, we obtain a new 1-handle (the central dotted circle) and the two 0-framed 2-handles. We can think of the 1-handle as connecting the top and bottom boundaries $I\times\Sigma\times \{0,1\}$ to each other, and each 2-handle connects a meridian of a dotted circle (essentially the core of the 1-handle) with its mirror image on the other boundary component.

We complete the analysis of Figure~\ref{C} by restoring the remaining curves to get $C$. At each twist box, we have a rationally canceling handle pair that represents a 2-handle added along a generator of $\Sigma$, with its core drilled out (the dotted circle). The $m$-framed circle represents the undrilled handle attached to  a product circle. Note that the diagram can be simplified by canceling the $m$-framed meridian, and when $r=s=1$, there is further cancellation at the twist boxes. (When $m=-1$ also, Figure~1 of Akbulut's recent preprint \cite{Acork} shows the result.)  According to the construction in \cite{InfCork}, the cork twist on $C$ is a twist on the torus $T=\{0\}\times\partial \Sigma\times S^1$ parallel to $\lambda=\{0\}\times\partial\Sigma\times\{\theta\}$. Interpreting $\lambda$ as the equator of $\partial B^3$ as before, we draw it as in Figure~\ref{T0}. (It encircles the clasp in $I\times \Sigma\times I$, but is drawn at one side to make room for the additional handles of $Y$.) Since the dual curve $\mu\subset T$ is a product circle in $Y$, the pair appears as in that same figure. Thickening these to annuli using the 0-framing, we obtain a punctured torus $T_0$ whose union with an embedded disk is $T$. We will verify directly from the diagrams that $\delta=\partial T_0$ bounds an embedded disk in $\partial C$ with interior disjoint from $T_0$, and that the $\lambda$-twist on the resulting torus has the required properties, but first we identify $C$:

\begin{figure}
\labellist
\small\hair 2pt
\pinlabel $-s$ at 95 108
\pinlabel $-s$ at 90 123
\pinlabel $-r$ at 154 133
\pinlabel $-r$ at 188 139
\pinlabel $m$ at 128 44
\pinlabel $0$ at 81 39
\pinlabel $0$ at 156 39
\pinlabel $\mu$ at 50 81
\pinlabel $\lambda$ at 0 86
\endlabellist
\centering
\includegraphics{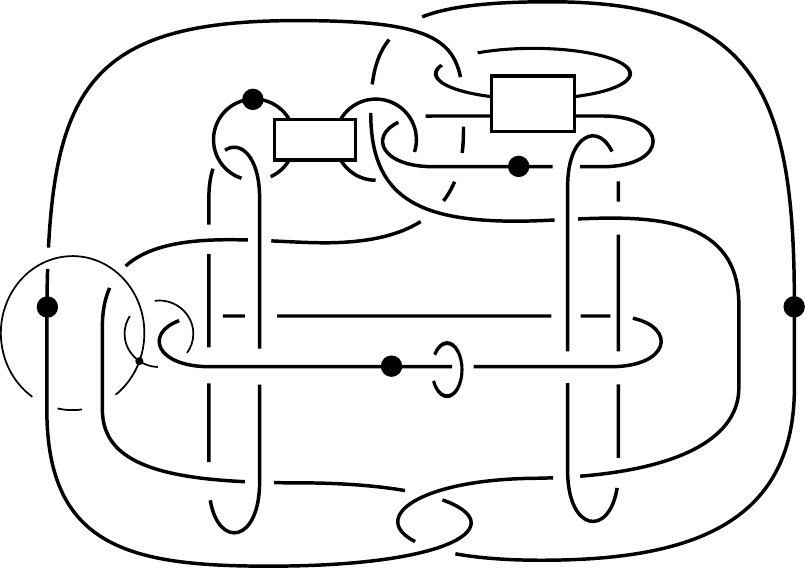}
\caption{The punctured torus $T_0$ in $\partial C$ is obtained by 0-framed thickening of $\lambda\cup\mu$. Its boundary is $\delta=\partial T_0$.}
\label{T0}
\end{figure}
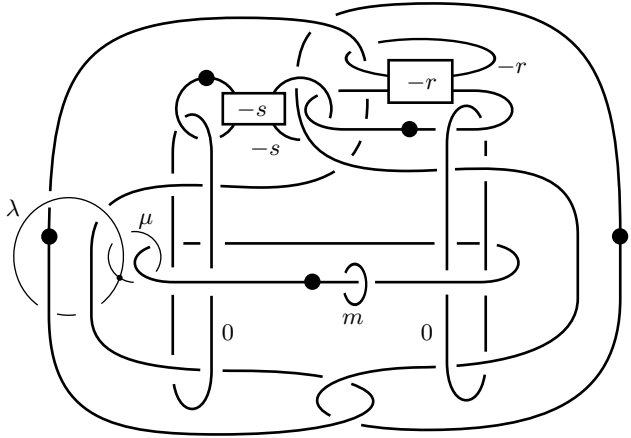

\begin{prop} The 4-manifold $C$ in Figure~\ref{C} is diffeomorphic to $C(r,s;m)$ in Figure~\ref{Crsm}.
\end{prop}

Note that this gives an independent check that $C$ is made from a double twist knot whose twists have opposite handedness. In \cite{InfCork}, this handedness was determined by a delicate inspection of 3-manifold orientations.

\begin{proof} Starting from Figure~\ref{C}, unwind the clasps of the two large dotted circles by moving the leftmost strand in the diagram through the large right circle and back to its place. After shortening the two 0-framed curves by an isotopy, we obtain Figure~\ref{5}. Next, we perform two double 1-handle slides as indicated. That is, each arrow represents two strands of a dotted circle being slid across another one. Throughout this proof, the link consisting of all dotted circles will be an unlink, so we are sliding 1-handles in the classical sense (ie,\ no nontrivial dotted ribbon links appear). We encounter a notational technicality: We slide using 0-framed parallel copies of the small dotted circles. These pass through negative twist boxes, so to restore the 0-framings we must add compensating positive twist boxes. The result, after the two obvious handle pair cancellations, is Figure~\ref{6}. Next, simplify the two large dotted curves by pulling the clasps through all twist boxes as indicated by the arrows, dragging along the dotted circle with the $m$-framed meridian. Raise the lowermost strand of the latter dotted circle so that it is positioned between the $-s$-framed 2-handle and the $\pm r$ twist boxes, then eliminate its self-crossing by flipping over the clasp running through the $\pm r$ twist boxes. The rightmost dotted circle can then be shrunk into the middle of the figure, which should then be recognizable as Figure~\ref{6half}. We next wish to cancel the $-s$-framed 2-handle. Since we cannot slide a dotted circle over a 2-handle, we first introduce a canceling 1-2 pair as in Figure~\ref{7}, then double slide the new 2-handle as indicated and cancel the $-s$-framed handle, obtaining Figure~\ref{8}. (Note that after the double slide, the new 2-handle initially runs twice through the $-s$-twist box, consistent with sliding over a $-s$-framed curve. However, we can immediately pull it down through the twist box to its position in Figure~\ref{8}.) Finally, slide the $-r$-twist box across the circle with the rightmost dot. This is a standard 3-manifold move, cf.\ Figure~\ref{1h} in Section~\ref{Twist}, obtained by repeatedly blowing up a $+1$-framed curve encircling the twist box, sliding it across the dotted circle, and blowing it back down. One way to interpret this move 4-dimensionally is to think of the dotted circles as representing $\# 3\thinspace S^1\times S^2$, containing a framed link. Perform the move on this 3-manifold, then uniquely fill in the 4-manifold $\natural 3\thinspace S^1\times D^3$. This move changes the framing on one 2-handle from $-r$ to 0, and that handle immediately cancels a dotted circle. The result is isotopic to the right-hand diagram in Figure~\ref{Crsm}.
\end{proof}

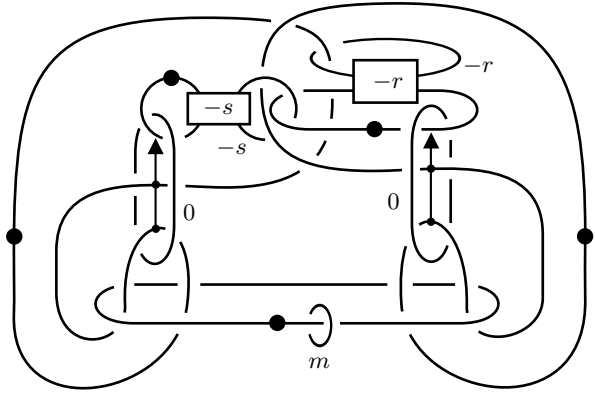
\begin{figure}
\labellist
\small\hair 2pt
\pinlabel $-s$ at 85 91
\pinlabel $-s$ at 80 106
\pinlabel $-r$ at 144 116
\pinlabel $-r$ at 178 122
\pinlabel $m$ at 118 10
\pinlabel $0$ at 69 67
\pinlabel $0$ at 146 71
\endlabellist
\centering
\includegraphics{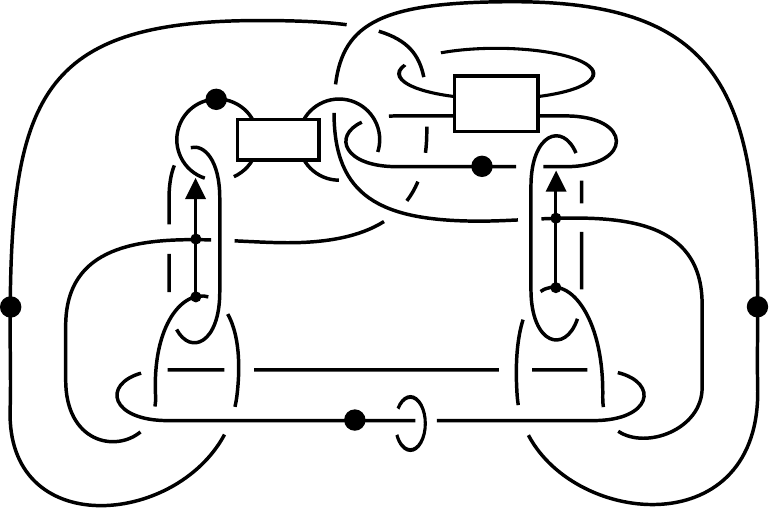}
\caption{Double 1-handle slides on a diagram of $C$ obtained from Figure~\ref{C} by isotopy.}
\label{5}
\end{figure}

\begin{figure}
\labellist
\small\hair 2pt
\pinlabel $-s$ at 77 85
\pinlabel $-s$ at 91 101
\pinlabel $s$ at 74 69
\pinlabel $r$ at 137 90
\pinlabel $-r$ at 159 95
\pinlabel $-r$ at 182 115
\pinlabel $m$ at 118 2
\endlabellist
\centering
\includegraphics{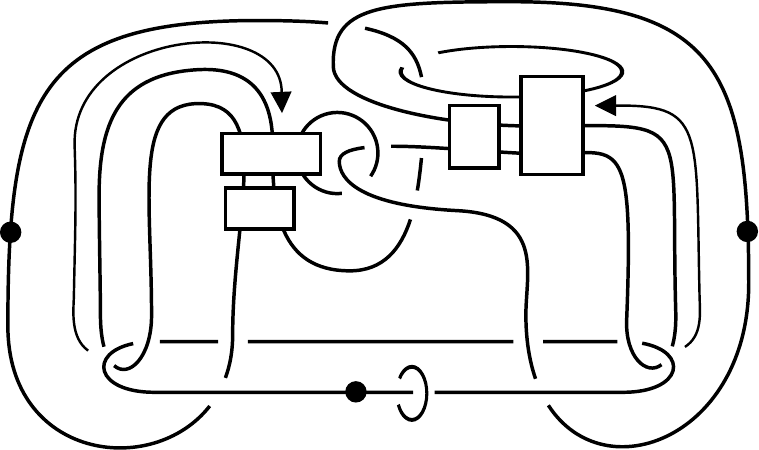}
\caption{Further simplifying $C$.}
\label{6}
\end{figure}

\begin{figure}
\labellist
\small\hair 2pt
\pinlabel $-s$ at 52 52
\pinlabel $-s$ at 68 37
\pinlabel $s$ at 49 36
\pinlabel $r$ at 131 62
\pinlabel $-r$ at 152 65
\pinlabel $-r$ at 197 84
\pinlabel $m$ at 132 29
\endlabellist
\centering
\includegraphics{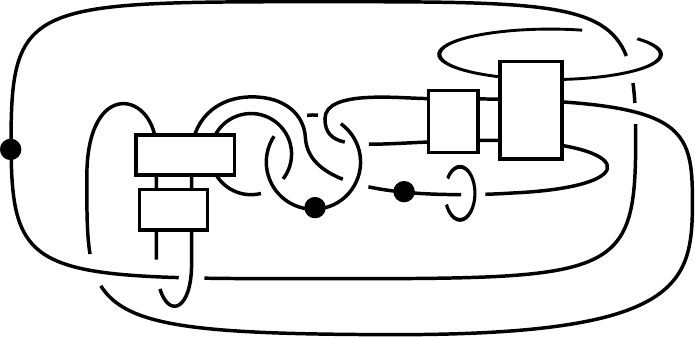}
\caption{An isotopic simplification of Figure~\ref{6}.}
\label{6half}
\end{figure}

\begin{figure}
\labellist
\small\hair 2pt
\pinlabel $-s$ at 52 60
\pinlabel $-s$ at 68 45
\pinlabel $s$ at 49 44
\pinlabel $r$ at 131 70
\pinlabel $-r$ at 152 73
\pinlabel $-r$ at 197 92
\pinlabel $m$ at 128 2
\pinlabel $0$ at 97 87
\endlabellist
\centering
\includegraphics{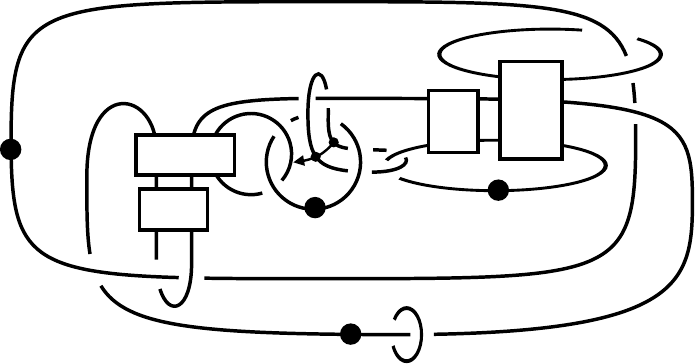}
\caption{A new 1-2 pair and a double slide.}
\label{7}
\end{figure}

\begin{figure}
\labellist
\small\hair 2pt
\pinlabel $-s$ at 25 69
\pinlabel $s$ at 26 41
\pinlabel $r$ at 106 73
\pinlabel $-r$ at 132 75
\pinlabel $-r$ at 173 92
\pinlabel $m$ at 102 2
\pinlabel $0$ at 57 67
\endlabellist
\centering
\includegraphics{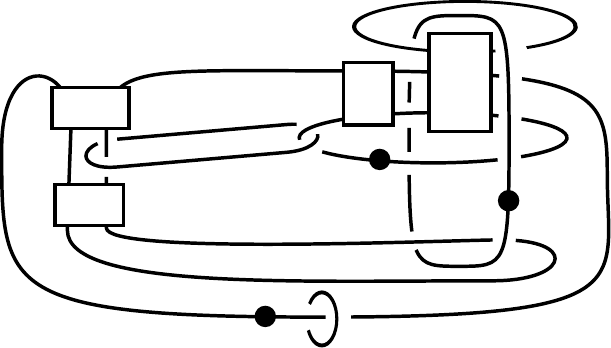}
\caption{Transferring $-r$ twists yields $C(r,s;m)$.}
\label{8}
\end{figure}

\begin{prop}\label{fk} The circle $\delta=\partial T_0\subset\partial C$ in Figures~\ref{T0} and~\ref{Z} bounds a disk $D$ in $\partial C$ with interior disjoint from $T_0$. The resulting torus twist parallel to $\lambda$ changes $k$ by 1 in Figure~\ref{Z} while otherwise preserving all curves in the figure (and their orientations).
\end{prop}

\begin{figure}
\labellist
\small\hair 2pt
\pinlabel $\eta_2$ at 269 277
\pinlabel $k$ at 39 178
\pinlabel $-k$ at 39 69
\pinlabel $-s$ at 148 247
\pinlabel $-s$ at 163 270
\pinlabel $-r$ at 247 262
\pinlabel $-r$ at 286 291
\pinlabel $\eta_1$ at 147 225
\pinlabel $\zeta_1$ at 85 252
\pinlabel $\zeta_2$ at 330 249
\pinlabel $m$ at 223 150
\pinlabel $+1$ at 190 21
\pinlabel $\rho$ at 317 131
\pinlabel $\sigma$ at 190 87
\pinlabel $\lambda$ at 9 125
\pinlabel $\mu$ at 79 174
\pinlabel $T_0$ at 66 134
\pinlabel $0$ at 131 68
\pinlabel $0$ at 255 68
\endlabellist
\centering
\includegraphics{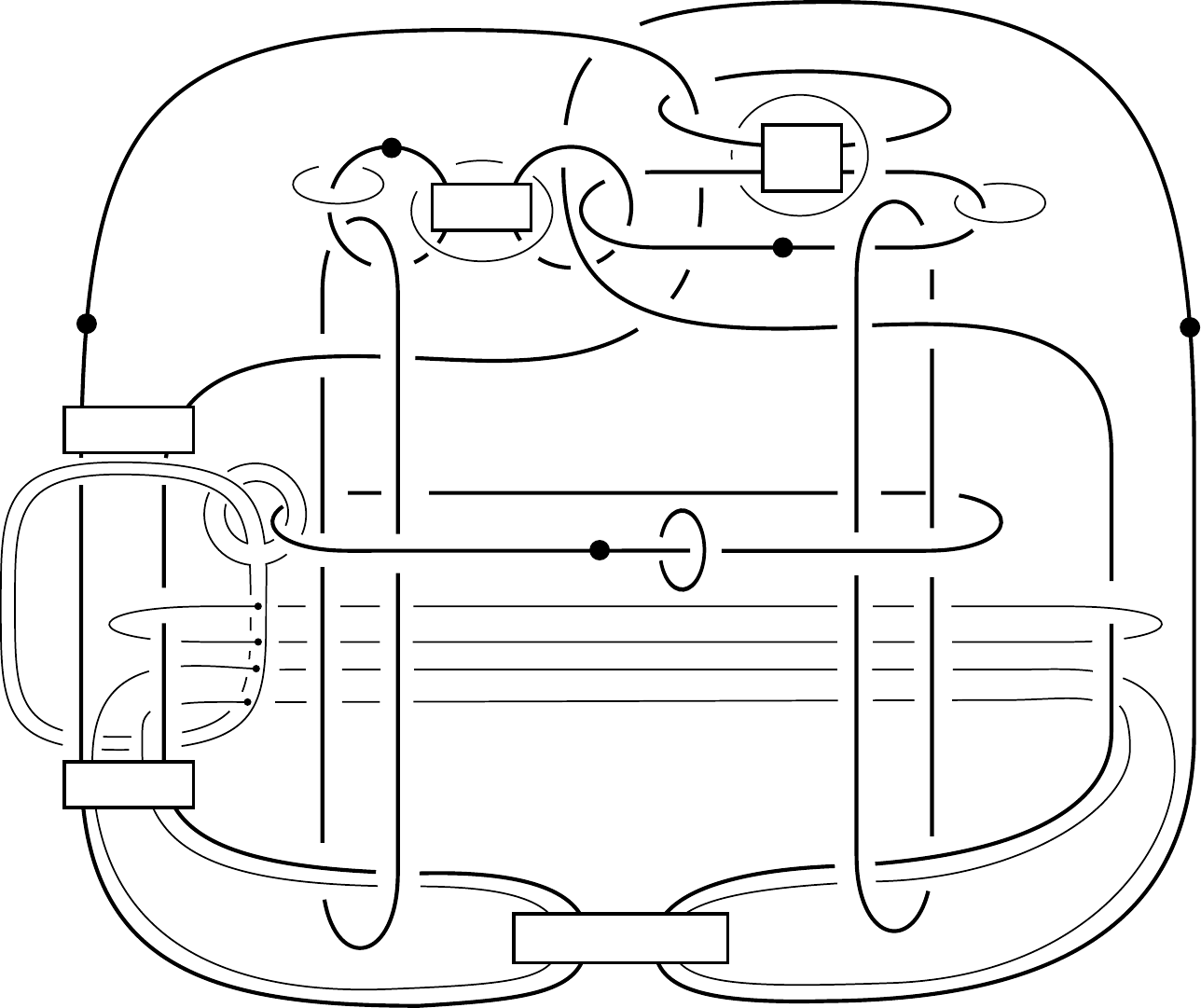}
\caption{The 4-manifold $Z_k(r,s;m)$ is obtained from $C(r,s;m)$ by adding 0-framed 2-handles along $\zeta_1$, $\zeta_2$ and $\rho$, then a 3-handle (and ignoring the other fine curves). Note that some curves intersect the punctured torus $T_0$, which is drawn explicitly as a thickening of the wedge $\lambda\cup\mu$. Its boundary $\delta$ is unlabeled.}
\label{Z}
\end{figure}
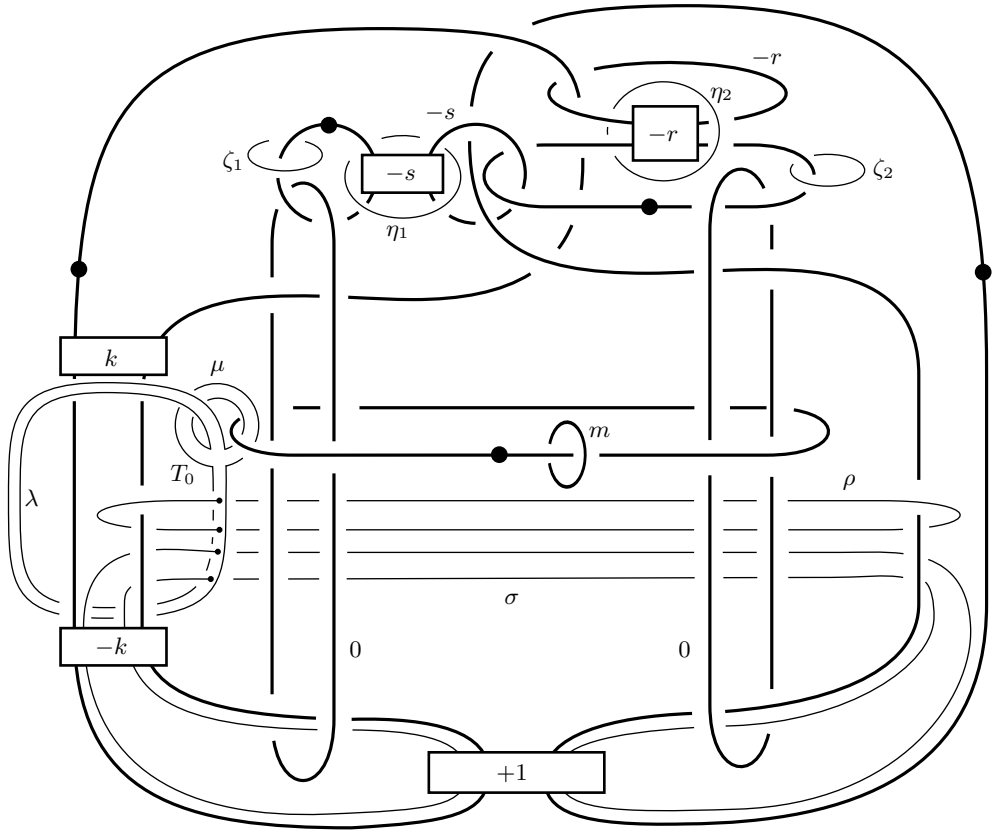

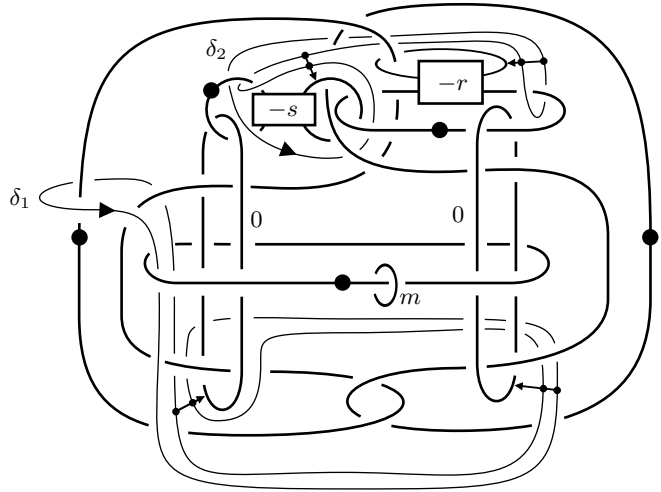
\begin{figure}
\labellist
\small\hair 2pt
\pinlabel $-s$ at 92 143
\pinlabel $-r$ at 156 153
\pinlabel $m$ at 140 72
\pinlabel $0$ at 82 102
\pinlabel $0$ at 158 104
\pinlabel $\delta_1$ at -7 110
\pinlabel $\delta_2$ at 67 166
\endlabellist
\centering
\includegraphics{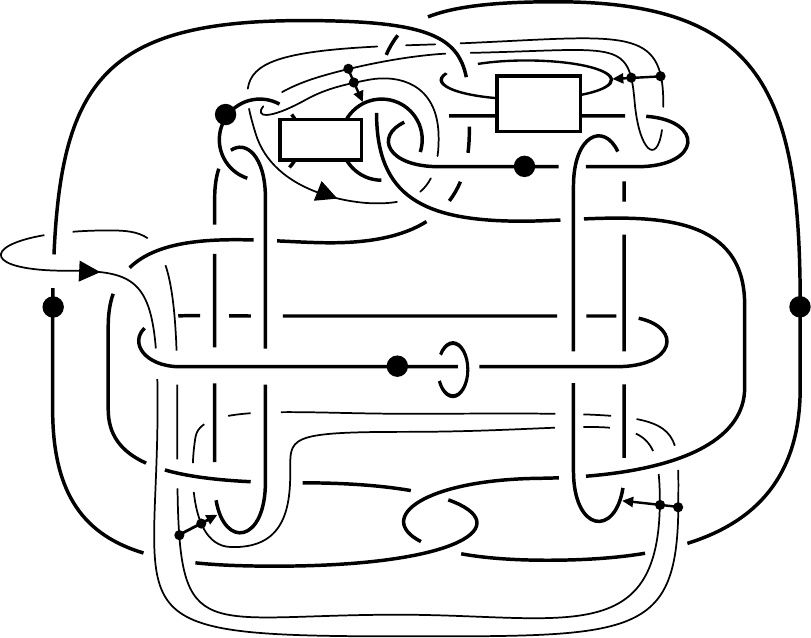}
\caption{A proof that $\delta$ bounds a disk disjoint from the fine curves of Figure~\ref{Z}.}
\label{disk}
\end{figure}

There are various approaches to the proof. A simple way to exhibit $D$ disjoint from the fine curves is by handle sliding $\delta$ to get an unknot, as in Figure~\ref{disk}. (Although this is unnecessary for proving our main theorem, an interested reader can verify that $\delta$ is isotopic to $\delta_1$. Two double handle slides change this to $\delta_2$, and two more yield an unknot. The boundary orientation of $\delta$ induced by the counterclockwise orientation of $T_0$ in Figure~\ref{Z} is shown as an aid.) We can make $D$ disjoint from $\inter T_0$ by 3-manifold theory (cf.\ proof of Theorem~\ref{equivalent}), but this could create intersections with the fine curves, resulting in their unexpected movement during the torus twist (cf.\ last paragraph of Section~\ref{Twist}.) This issue can presumably be dealt with, but we instead prove the theorem definitively with a direct approach.

\begin{proof} Set $k=0$, and drag $T_0$ and all auxiliary curves in Figure~\ref{Z} simultaneously through the computation of the previous proof. This is routine but tedious; details are left to the intrepid reader. (One can treat $T_0$ as a framed wedge of circles, as long as its intersections with the auxiliary curves are handled carefully. These intersections will eventually be dragged through the $-s$-twist box, in the downward direction. Note that the curves $\eta_i$ will remain closely encircling the negative twist boxes, so need not be carefully tracked; the curves $\zeta_i$ will be similarly rooted to the positive twist boxes as soon as these boxes appear.) The result is Figure~\ref{TinC}, where the curves $\eta_i$ and $\zeta_i$ encircling the twist boxes are suppressed, and $T_0$ is the obvious 0-framed thickening of $\lambda\cup\mu$. (We have drawn the thickening near where $T_0$ intersects the other curves.) We see that $T_0$ is as originally described in $C(r,s;m)$, with $\lambda$ the longitude of $\kappa(-r,s)$ and $\mu$ a meridian, and the disk $D$ easily visualized. The remaining curves can be explicitly seen to avoid $T$ except for the original intersections of $T$ with $\rho$ and $\sigma$. The torus twist $f$ wraps these curves parallel to $\lambda$ at the intersections, and fixes all curves elsewhere. When we transport this description back to Figure~\ref{Z} with $k=0$, we can undo the wrapping caused by $f^k$ by an isotopy that restores the twist boxes to their original values $\pm k$.
\end{proof}

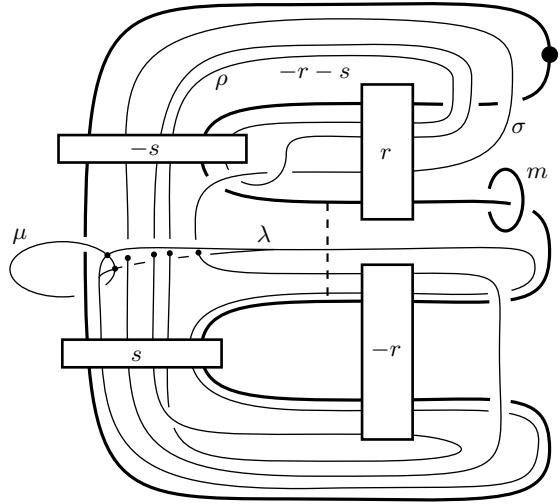
\begin{figure}
\labellist
\small\hair 2pt
\pinlabel $-r-s$ at 115 162
\pinlabel $\rho$ at 80 158
\pinlabel $\sigma$ at 192 141
\pinlabel $m$ at 199 124
\pinlabel $\mu$ at 4 100
\pinlabel $\lambda$ at 96 101
\pinlabel $-s$ at 51 133
\pinlabel $r$ at 142 131
\pinlabel $s$ at 48 55
\pinlabel $-r$ at 142 57
\endlabellist
\centering
\includegraphics{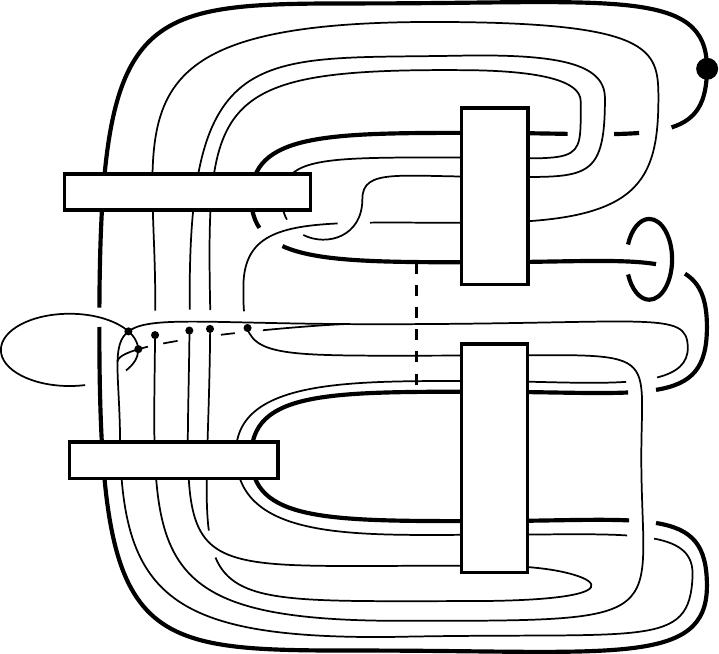}
\caption{A complete picture of the torus $T\subset\partial C$, with the fine curves of Figure~\ref{Z}. The generating circles of $T$ are $\mu$ and $\lambda$, the latter of which has been partially thickened in $T$ to show the intersections of $T$ with the fine curves. The rest of $T$ is given by the horizontal plane of symmetry (and point at infinity), surgered by a tube following the lower half of the dotted circle so that $T$ contains  $\mu$ and $\lambda$.}
\label{TinC}
\end{figure}

Figure~\ref{Z} was drawn so that the case $r=s=-m=1$ exhibits Akbulut's ``$\delta$-curve" \cite[Figure~4]{Acork} as the boundary of $T_0$. We will show in Section~\ref{Twist} that under broad hypotheses, every $\delta$-curve arises from a torus in this manner. The author's first diagrammatic proof used a different approach: Blow up a $(-1)$-framed unknot parallel to $\lambda$ in the lower half of Figure~\ref{Z} (encircling the $-k$-twist box), slide it around $T$ to get a curve encircling the $+k$-twist box using a diagram similar to Figure~\ref{disk}, then blow it back down. While this is the same torus twist (cf.\ also Figure~\ref{alpha} and its discussion in the last paragraph of Section~\ref{TorusTwists}), it is less clear from this method that the implicitly described torus is embedded. However, the diffeomorphism is still well-defined by this procedure and has the required properties. This easier method is strong enough for all of our subsequent discussion, since we do not need to exhibit the diffeomorphism as a twist on an embedded torus.

Now consider the 4-manifold obtained from $C$ by attaching 0-framed 2-handles along the fine curves $\zeta_1$, $\zeta_2$ and $\rho$ shown in Figure~\ref{Z}. We will show that its boundary contains a nonseparating 2-sphere. Add a 3-handle along this sphere and call the result $Z_k(r,s;m)$. This is independent of the choice of 2-sphere (by Trace \cite{Tr}, for example), and canonically contains a copy of $C$. Another picture of $Z_k(r,s;m)$ is given by Figure~\ref{Z2}, where we have switched to dotted ribbon knot notation and canceled the new 2-handles $\zeta_1$, $\zeta_2$ and $\rho$ against 1-handles. (Ignore the curve $\sigma$ in Figure~\ref{Z2} but include the other fine curves, which come from 2-handles in Figure~\ref{Z}, and the 3-handle.) Canceling $\rho$ has joined the two large dotted circles, forming the knot $K_k\#-K_k$ where $K_k$ is the twist knot $\kappa(k,-1)$. This dotted ribbon knot represents $I\times E(k,1)$, cf.\ Figure~\ref{Crsm}. (The comparison with $I\times E(r,s)\subset C(r,s;m)$ is superficial.) Without the fine curves, Figure~\ref{Z2} represents the manifold $W_k=S^1\times E(k,1)$. (Each handle of $I\times E(k,1)$ generates an additional handle of index higher by one in $W_k$, cf.\ Figure~\ref{C}, with the canceled 2-handle $\rho$ generating the 3-handle.) The fine curves in Figure~\ref{Z2}, two of which are isotopic, represent the product $S^1$ ($m$-framed), meridians of $K_k$ ($-r$-framed  and $-s$-framed), and its longitude $\sigma$ in $\{0\}\times E(k,1)$. Unlike previously, this recognition of $W_k$ is crucial to our proof, so we provide a check:

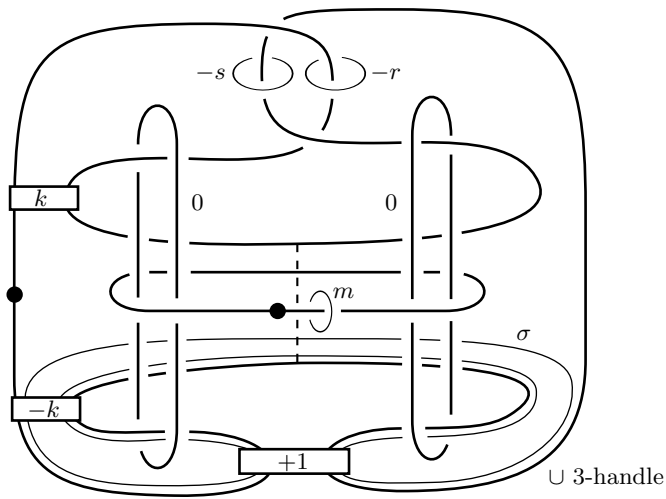
\begin{figure}
\labellist
\small\hair 2pt
\pinlabel $-s$ at 77 160
\pinlabel $-r$ at 143 160
\pinlabel $m$ at 127 77
\pinlabel $0$ at 72 111
\pinlabel $0$ at 145 111
\pinlabel $k$ at 13 113
\pinlabel $-k$ at 14 33
\pinlabel $\sigma$ at 195 61
\pinlabel $+1$ at 108 13
\pinlabel{ $\cup$ 3-handle} at 226 8
\endlabellist
\centering
\includegraphics{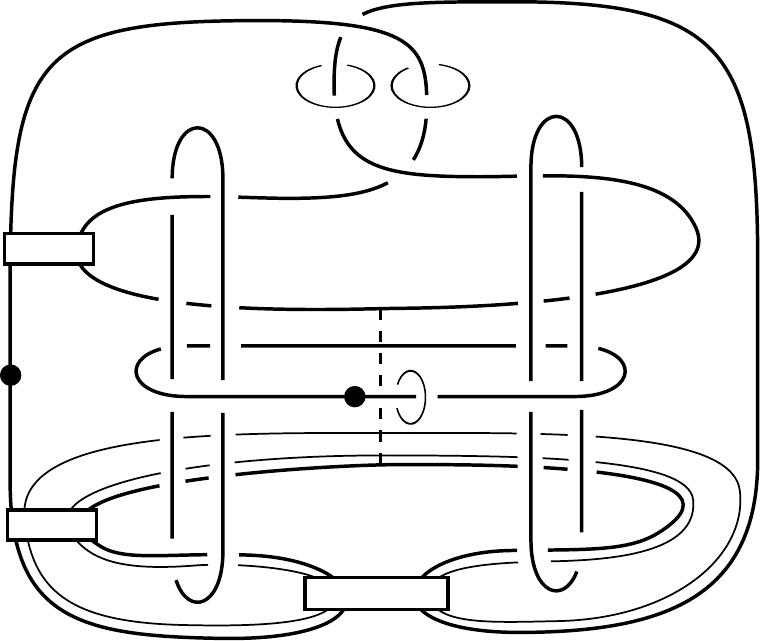}
\caption{Another picture of $Z_k(r,s;m)$, with extra curve $\sigma$.}
\label{Z2}
\end{figure}

\begin{prop}\label{Wk} The manifolds $Z_k(r,s;m)$ and $W_k$ given by the diagrams are well-defined (ie, the relevant 3-manifolds have a nonseparating sphere for the 3-handle). The 3-manifolds $\partial W_k$ are all diffeomorphic, preserving the fine curves of Figure~\ref{Z2} and their orientations. The 4-manifold $W_0$ is diffeomorphic to $T^2\times D^2$, with $\sigma$ bounding the essential disk, and the $-r$- and $m$-framed curves arising as factors of a product decomposition $T^2=S^1\times S^1$.
\end{prop}

\begin{proof} Interpreting Figure~\ref{Z2} as a 3-manifold (ignoring the 3-handle), we can eliminate the clasps from the dotted ribbon knot by blowing up a $+1$-framed unknot as in Figure~\ref{unclasp}, sliding this unknot over the tall unknots as shown, and blowing back down. We can now cancel the $\pm k$ twist boxes by twisting one tall unknot $k$ times about its long axis. This identifies each $\partial W_k$ with $\partial W_0$ (and similarly for $Z_k(r,s;m)$) and reduces well-definedness to the $k=0$ case. Now consider Figure~\ref{Z2} to be a 4-manifold. When $k=0$ the $\pm k$-twist boxes can be erased, so we can pull the outer strand of $\sigma$ through the $+1$-twist box and unwind the outer strands of the large dotted circle as in Figure~\ref{W0}. To get this figure, we also swing both tall curves to the inner rear of the large dotted circle. They are then parallel, so one can be slid over the other to become a 0-framed unknot unlinked from the rest of the diagram. This exhibits the nonseparating sphere in $\partial W_k$. Canceling this unknot with the 3-handle, we obtain Figure~\ref{W0}, which is the Borromean rings with  fine meridians of each component. This has the required interpretation.
\end{proof}

\begin{figure}
\labellist
\small\hair 2pt
\pinlabel $-s$ at 75 177
\pinlabel $-r$ at 143 177
\pinlabel $m$ at 120 77
\pinlabel $0$ at 72 111
\pinlabel $0$ at 145 114
\pinlabel $k$ at 13 113
\pinlabel $-k$ at 14 33
\pinlabel $\sigma$ at 195 61
\pinlabel $+1$ at 108 13
\pinlabel $0$ at 225 101
\pinlabel $0$ at 184 83
\pinlabel $+1$ at 80 163
\endlabellist
\centering
\includegraphics{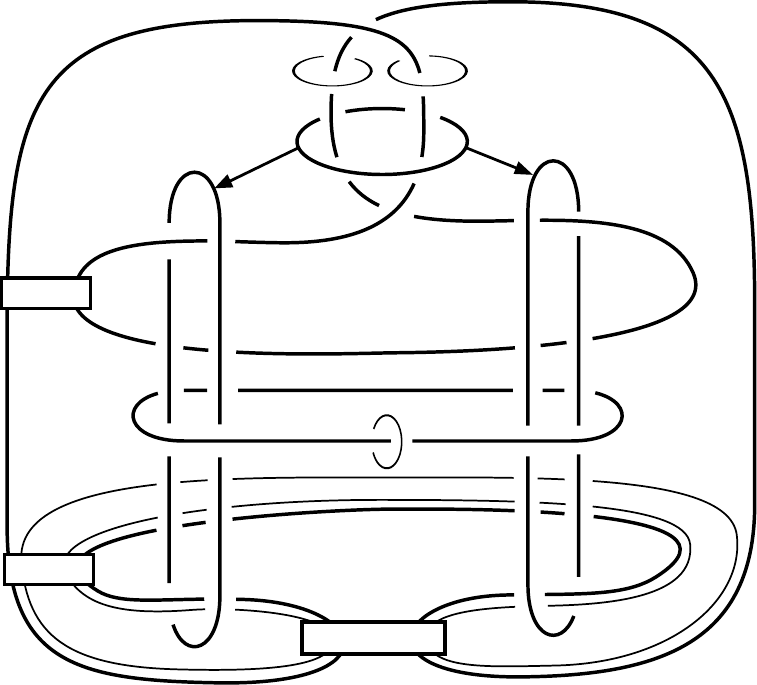}
\caption{Simplifying $\partial W_k$.}
\label{unclasp}
\end{figure}
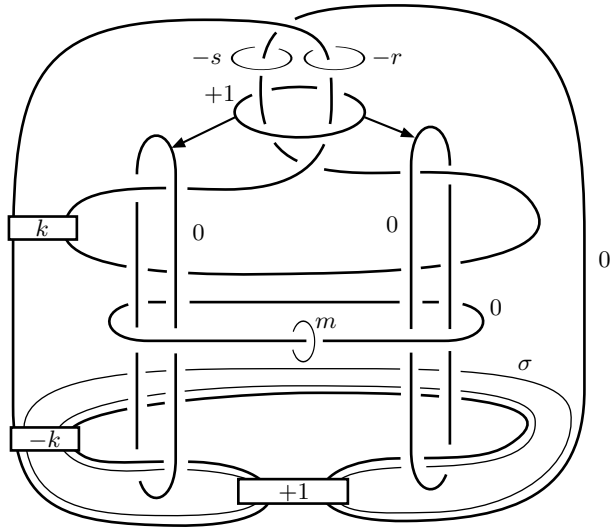

\begin{figure}
\labellist
\small\hair 2pt
\pinlabel $-s$ at 76 160
\pinlabel $-r$ at 144 160
\pinlabel $0$ at 133 114
\pinlabel $m$ at 89 60
\pinlabel $\sigma$ at 195 53
\endlabellist
\centering
\includegraphics{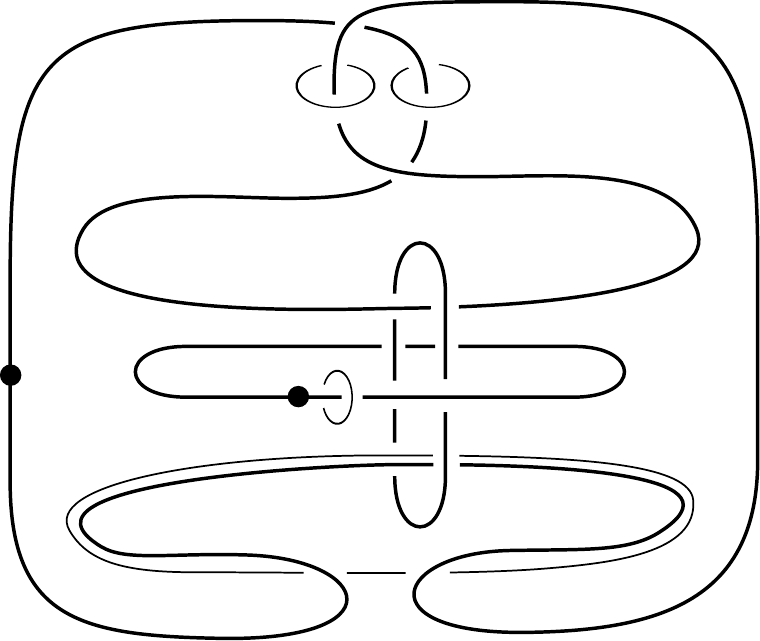}
\caption{Identifying $W_0$ as $T^2\times D^2$.}
\label{W0}
\end{figure}

We can now identify the manifolds $Z_k(1,1;-1)$ for all $k\in\Z$. First, $Z_0(r,s;m)$ is obtained from $W_0=T^2\times D^2$ by adding three 2-handles along embedded circles in copies of $T^2\times\{p\}$, as shown in Figure~\ref{elliptic}. (The 1-handles and 0-framed 2-handle exhibit $T^2\times D^2$ so that the trivial torus bundle structure on its boundary is easily visible.) When $r=s=1$ and $m=-1$, this diagram is a well-known description of an elliptic fibration over a disk, a cusp neighborhood with an extra vanishing cycle, eg, \cite[Section 8.2]{GS}. Thus, $Z_0(1,1;-1)$ naturally embeds in the elliptic surface $E(n)$ for any fixed $n>0$, and is easily seen in link diagrams of the latter. The curve $\sigma$ is a section of the induced torus bundle structure on $\partial Z_0(1,1;-1)$ (as can again be seen in Figure~\ref{elliptic}, cf.\ \cite{GS}). For general $k$, $Z_k(r,s;m)$ is obtained from $Z_0(r,s;m)$ by replacing $W_0=T^2\times D^2$ by $W_k$, preserving the longitude $\sigma$.   This is precisely the Fintushel-Stern knot construction, using the knot $K_k$, and (when $r=s=1$ and $m=-1$) using a regular fiber of the elliptic fibration on $Z_0(1,1;-1)$.

\begin{figure}
\labellist
\small\hair 2pt
\pinlabel $-s$ at 37 50
\pinlabel $-r$ at 37 81
\pinlabel $m$ at 73 92
\pinlabel $\sigma$ at 139 30
\pinlabel $0$ at 117 119
\endlabellist
\centering
\includegraphics{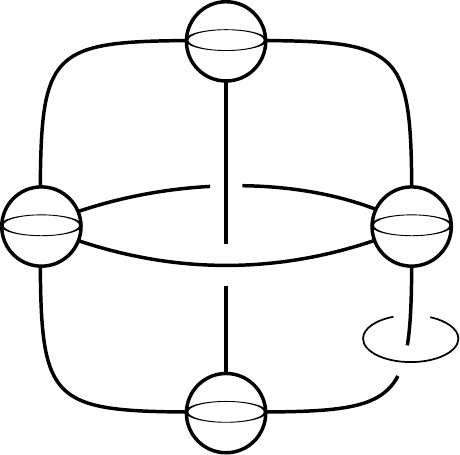}
\caption{$Z_0(r,s;m)$ showing elliptic fibration of $Z_0(1,1;-1)$ with section $\sigma$.}
\label{elliptic}
\end{figure}
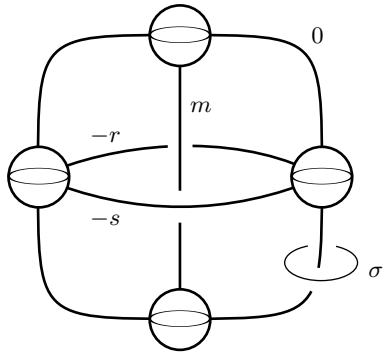

To prove our main theorem, we need one last routine lemma:

\begin{lem}\label{Zdiff} A self-diffeomorphism $\varphi$ of the pair  $(\partial Z_0(1,1;-1),\sigma)$ that preserves the orientation of $\sigma$ must be isotopic to the identity (through self-maps of the pair).
\end{lem}

The control of $\sigma$ is necessary, in order to rule out twists on fiber tori. The proof rules out horizontal tori.

\begin{proof} We can identify each fiber of the torus bundle with $\R^2/\Z^2$ so that $\sigma$ is zero in each fiber. Then the monodromy is an element $A\in\SL(2,\Z)$. If we use the obvious basis for $\R^2$ in Figure~\ref{elliptic}, then $A$ is given by $\pi/2$ rotation. (This is both well-known and routine to verify in the figure. See also \cite{GS}.) Choose a fiber $F$ and assume its image $\varphi(F)$ is transverse to it. Since $F$ is incompressible, each circle of intersection is trivial in $F$ if and only if it is trivial in $\varphi(F)$. Each innermost circle in $\varphi(F)$ also bounds a disk in $F$. The two disks together bound a ball $B$. If one disk intersects $\sigma$ (necessarily in a unique point), then so does the other, and $\sigma\cap B$ is an unknotted arc in $B$. (Otherwise, the complement of a lift of $\sigma$ to the universal cover of $\partial Z_0(1,1;-1)$ would have nonabelian fundamental group.) Either way, we can eliminate trivial circles by isotoping $\varphi$ pairwise until all circles (if any remain) are essential. These must be parallel to each other in both of the tori $F$ and $\varphi(F)$, cutting each into annuli. Since $\pm 1$ is not an eigenvalue of $A$, no such annulus of $\varphi(F)$ can surject onto the base circle. Thus, if the intersection is nonempty, $\varphi(F)$ has more than one annulus, and we can choose one that does not contain the intersection point with $\sigma$. This fits together with an annular region in $F$ that we choose disjoint from $\sigma$, to form a nullhomologous torus. This torus is compressible (as seen, for example, in the $\Z$-cover of $\partial Z_0(1,1;-1)$), so it bounds a solid torus, with the annuli bounded by longitudes of it. We can now reduce the number of intersection circles until $\varphi(F)$ is disjoint from $F$. Cutting along $F$, we see $\varphi(F)$ as an incompressible torus in $F\times I$. Applying standard theory to the complement of $\sigma$, eg, Waldhausen \cite[Proposition~3.1]{Wald}, we can arrange $\varphi(F)$ to be a fiber, and then isotope $\varphi$ so that it covers the identity map on the base circle. It is easily checked that the only elements of $\SL(2,\Z)$ that commute with $A$ are powers of $A$. Thus, on each fiber, $\varphi$ restricts to $A^j$ for a fixed $j$. We can then change $\varphi$ to the identity by a fiber-preserving isotopy covering a $2\pi j$-rotation on the base.
\end{proof}

We can now prove our main theorem. For $k\in\Z$ and fixed $r,s,n>0>m$, let $X_k$ be the 4-manifold obtained from $E(n)$ by the Fintushel-Stern construction on a fiber, using the twist knot $K_k$ (so $X_0=E(n)$). We have embeddings $C(r,s;m)\subset Z_0(r,s;m)\subset X_0\# N\overline{\CP^2}$, where $N=r+s-m-3\ge 0$, and the last embedding is obtained from the simplest case $r=s=-m=1$ by blowing up meridians of the three negatively framed 2-handles of Figure~\ref{elliptic} to suitably lower their framings. Let $X^*_k$ be the manifold obtained from $X_0\# N\overline{\CP^2}$ by cutting out $C(r,s;m)$ and regluing it via the torus twist $f^k$.

\begin{thm} \cite{InfCork} For each $k$, the manifold $X^*_k$ is diffeomorphic to $X_k\# N\overline{\CP^2}$. In particular, the manifolds $X^*_k$ for $k\in\Z$ are pairwise nondiffeomorphic, so $(C(r,s;m),f)$ is an infinite order cork (for each fixed choice of $r,s,m$ as above).
\end{thm}

\begin{proof} First consider the simplest case $r=s=-m=1$. Starting from the embedding $Z_0(1,1;-1)\subset X_0$, we can cut out a regular neighborhood $W_0$ of a fiber inside $Z_0(1,1;-1)$ and replace it by $W_k$, obtaining  $Z_k(1,1;-1)\subset X_k$ with the embedding preserving $\sigma$ (Proposition~\ref{Wk} and below). Alternatively, we can cut out the cork $C(1,1;-1)$ and reglue it by $f^k$, obtaining $Z_k(1,1;-1)\subset X^*_k$, again preserving $\sigma$ (Proposition~\ref{fk}). Clearly, the complements of $Z_k(1,1;-1)$ in these two closed manifolds are identified (preserving $\sigma$). But the two embeddings of  $Z_k(1,1;-1)$ are related by a diffeomorphism preserving $\sigma$ and its orientation, so by Lemma~\ref{Zdiff} we can assume they agree on the identified boundaries of the complements. Thus, the diffeomorphisms fit together as required.

For the general case, we blow up to obtain embeddings $Z_0(r,s;m)\subset Z_0(1,1;-1)\# N\overline{\CP^2}\subset X_0\# N\overline{\CP^2}$. The first embedding is obtained from Figure~\ref{Z} by adding $-1$-framed meridians to the curves with framings $-r$, $-s$ and $m$ so that blowing down changes all three framings to $-1$. After the two 2-handles $\zeta_i$ of $Z_0(r,s;m)$ cancel their 1-handles, two sets of these $-1$-framed meridians can be drawn as parallel copies of the curves $\eta_i$. Since the torus twist does not disturb $\eta_1$, $\eta_2$ or the $m$-framed meridian, it gives embeddings $Z_k(r,s;m)\subset Z_k(1,1;-1)\# N\overline{\CP^2}\subset X^*_k$. The theorem now follows from Lemma~\ref{Zdiff} as in the previous case.
\end{proof}

In principle, there should be a direct proof of the theorem, by drawing $X^*_k$ and $X_k\# N\overline{\CP^2}$, and exhibiting an explicit diffeomorphism. A link diagram of $X_k$ was drawn by Akbulut, then independently produced as \cite[Figure~10.2]{GS} (discussion on pp.\ 407--8), using the technique of \cite{AkScharl}. This diagram is obtained from Figure~\ref{Z} of $ Z_k(1,1;-1)$ by adding some 2-handles and a 4-handle. One 2-handle is attached along $\sigma$ with framing $-n$. The others are $-1$-framed and attached along parallel copies of the circles with framing $m$ and $-r$ (or $-s$), but the two types of new curves are interleaved. A diagram of $X^*_k$ can be similarly constructed by torus twisting $ Z_0(1,1;-1)$. To show the diagrams are diffeomorphic, it suffices to connect the $-r$- and $m$-framed curves by a framed arc whose union with the two attached circles and $\sigma$ is preserved (after handle slides) by the torus twist, since all the new 2-handles will be attached in a neighborhood of these. This project has not been attempted with sufficient intensity for success.


\section{Twists that preserve 4-manifolds}\label{Fishtail}

Having explicitly exhibited infinite order cork twists, we now address the opposite issue, finding conditions under which twisting a contractible submanifold does not change the diffeomorphism type of a 4-manifold. Let $T\subset M$ be an embedded torus or Klein bottle in a 3-manifold, and let $f$ be  the twist on $T$ parallel to a circle $\alpha\subset T$, as described in Section~\ref{TorusTwists}. Let $W$ be the elementary cobordism built from $I\times M$ by adding a 2-handle $h$ to $\{1\}\times M$ along a parallel copy $\gamma$ of $\alpha$, with framing $\pm1$ relative to $T$. Thus, the top boundary $\partial_+W$ is obtained from $M$ by surgery on $\gamma$.

\begin{thm}\label{cobordism} The twist $f$ on $\partial_-W=\{0\}\times M$ extends over $W$ so that it is the identity on $\partial_+W$.
\end{thm}

\begin{proof} Let $g_t$ be an isotopy of the identity on $M$, supported in a tubular neighborhood of $T$, that preserves $T$ setwise but rotates it once parallel to a circle $\beta$ dual to $\alpha$. Interpret the isotopy $g_t\circ f$ as a self-diffeomorphism of $I\times M$. We can assume that $\gamma$ lies outside the support of this map in  $\{1\}\times M$, then extend over the handle $h$ by the identity. In $\partial_+W$, Figure~\ref{alpha} (reversed) shows an isotopy rotating $T$ back to its original position while undoing the twist produced by $f$.
\end{proof}

In the case where $T$ is a torus, the above diffeomorphism of $W$ is a manifestation of a {\em fishtail twist}. The latter has been used in various forms for some decades; see \cite{CS} for a recent discussion. If $N$ denotes a tubular neighborhood of $T$ in $M$, then $I\times N\approx T^2\times D^2$, and $I\times N\cup h$ is a fishtail neighborhood. It is well known that the twist on $\{0\}\times T$ parallel to $\alpha$ extends over this neighborhood as the identity on the rest of its boundary. The main point is that the boundary is a torus bundle with monodromy given by a Dehn twist parallel to $\alpha$, so the torus twist can be absorbed by a fiber-preserving isotopy covering a full rotation of the base.

As an application, we partially answer \cite[Question~1.6]{InfCork}. Let $D\subset B^4$ be a slice disk for a composite slice knot $K=\partial D$. For example, $D$ can be the the obvious ribbon disk for any nontrivial knot of the form $\kappa\#-\kappa$, the case considered in \cite{InfCork}. Let $C=C(D,m)$ be the contractible 4-manifold obtained from the slice complement by adding a 2-handle along a meridian with framing $m\ne0$, so that $C$ is $C(r,s;m)$ in the case $\kappa=\kappa(r,-s)$. The boundary of $C$ is the homology sphere obtained by $(-\frac1m)$-surgery on $K$. It is irreducible and has incompressible tori as in the previous section: Start with a sphere $S$ in $S^3$ intersecting $K$ in two points and splitting it nontrivially as a sum $K_0\# K_1$. Remove the intersections by surgering $S$ to a torus in $S^3-K$, using a tube following $K_1$. Such a torus has an obvious product decomposition, with one factor a meridian of $K$ and the other a 0-framed longitude of $K_1$. The cork twists of $C(r,s;m)$ in the previous section have this form for a longitudinal twist, on the unique incompressible torus if $m=-1$. It was asked in \cite{InfCork} whether twisting a fixed embedding of $C$ by the full action of such a torus could give a family of distinct diffeomorphism types indexed by $\Z\oplus\Z$. Previously, a preliminary version of Akbulut's 2014 posting \cite{withdrawn} unsuccessfully attempted to show that the meridian twist was an infinite order cork twist in the case of the obvious ribbon disk for the square knot with $m=-1$, that is, the manifold $C(1,-1;-1)$ in our present notation. However, both constructions are impossible when $m=\pm1$:

\begin{cor}\label{trivial} Every torus twist parallel to the meridian of $K$ extends over $C(D,\pm 1)$. In particular, for each $r,s\in\Z$, the meridian twist on $\partial C(r,s;\pm 1)$ extends over $C(r,s;\pm 1)$.
\end{cor}

\begin{proof} We find a cobordism $W\subset C$ as in the theorem with $\partial_-W=\partial C$. Since the diffeomorphism extends over $W$ as the identity on $\partial_+W$, we can extend as the identity over the rest of $C$. To construct $W$, begin with a collar of $\partial C$. The additional 2-handle $h$ is obtained by thickening the cocore of the meridian 2-handle $h^*$ of $C$. The attaching circle of $h$ is a 0-framed meridian to that of $h^*$. Interpreting the diagram as a 3-manifold and blowing down $h^*$, we realize $h$ by a $\mp 1$-framed meridian of $K$ as required.
\end{proof}

It follows that for a fixed embedding and torus, cutting $C(D;\pm 1)$ out of a 4-manifold and regluing it by torus twists generates a family of 4-manifolds whose diffeomorphism types are indexed at most by $\Z$. The problem remains open when $|m|>1$. However, distinguishing  meridian twists of $C(r,s;m)$ would require a somewhat different approach, since the proof in \cite{InfCork} depends on an embedding in a 4-manifold $X$ satisfying the hypothesis of the following for the meridian $\alpha$:

\begin{cor} Let $Y\subset X$ be a 4-manifold pair, and let $f$ be a twist on a torus or Klein bottle $T\subset\partial Y$, parallel to some curve $\alpha$. Suppose that $X-\inter Y$ contains an embedded disk with boundary $\alpha$, inducing framing $\pm 1$ relative to $T$. Then cutting out $Y$ and regluing it after twisting by a power of $f$ yields a manifold diffeomorphic to $X$.
\end{cor}

\begin{proof} Observe the cobordism $W$ in $X-\inter Y$. Extend $f^k$ outward from there by the identity.
\end{proof}

The same question of \cite{InfCork} asks about longitudinal twists for slice disks not covered by the main theorem of that paper. Ray and Ruberman \cite{RR} have recently observed that when $K_1$ is a torus knot, {\em every} twist on the torus determined by $K_1$ extends over $C(D,\pm1)$. This is seen by combining  Corollary~\ref{trivial} with the Seifert circle action on the complement of $K_1$, which shows that twisting on some (nonzero) longitude is isotopic to the identity. A closer look yields the first examples of contractible manifolds, including $C(1,-1;-1)$, that cannot be nontrivial corks even though their boundaries have incompressible tori:

\begin{cor}\label{noncork} Let $C=C(D,\pm1)$ be obtained as above with $\partial D=\kappa\#-\kappa$, where $\kappa$ is a torus knot. Then every diffeomorphism of $\partial C$ extends over $C$.
\end{cor}

\begin{proof} The boundary of $C$ is $\mp 1$-surgery on $\kappa\#-\kappa$, which is made by gluing together the complements of $\kappa$ and $-\kappa$ along their boundary tori $T$. (The surgery can be interpreted as a twist in the gluing map.) Since $T$ is the unique incompressible torus in $\partial C$, it is preserved by any self-diffeomorphism (up to isotopy). We have seen that twists on $T$ extend. The nonzero signature of $\kappa$ rules out orientation-preserving diffeomorphisms of $\partial C$ switching the two knot complements, and orientation-reversing switches are ruled out by the handedness of the gluing map. The complement of $\kappa$ is Seifert fibered with two exceptional fibers, and any self-diffeomorphism preserves this structure, so we are left with only the involution of $\kappa$ that reverses its string orientation. The induced involution of $\partial C$ obviously extends.
\end{proof}

\section{Torus twists and $\delta$-moves}\label{Twist}

In this section, we give a careful definition of Akbulut's $\delta$-moves, and almost entirely reduce them to twists on tori, Klein bottles and spheres. Torus and Klein bottle twists were introduced in Section~\ref{TorusTwists}. Twists on spheres are defined similarly but have order at most 2 in $\pi_0(\Diff_+(M))$ (since $\pi_1(\SO(3))=\Z/2$). Klein bottle twists are of limited use: They only arise when $M$ contains the $I$-bundle over the Klein bottle with orientable total space, a somewhat rare phenomenon. In particular, this does not occur for homology spheres, so there can be no Klein bottle cork twists. A tubular neighborhood of a Klein bottle $K$ is bounded by a torus $T$ double covering $K$. This is incompressible if and only if $K$ is, since any compressing disk for $K$ must be bounded by an orientation-preserving loop in $K$. It is easy to see that the square of a twist on $K$ is a twist on $T$ parallel to the same circle $\alpha$.

We define $\delta$-moves following Akbulut \cite{Acork}, with additional attention to detail in anticipation of the upcoming proofs. First, consider the standard 3-manifold diffeomorphism given by Figure~\ref{1h}, which can be obtained by blowing up a $\pm1$-framed unknot around one twist box, sliding this unknot over the 0-framed circle so that it surrounds the other twist box, then blowing back down. For a $\delta$-move, start with a framed circle $C$ in a 3-manifold $M$. Draw $M$ as surgery on a link $L$ in $S^3$ so that $C$ appears as a 0-framed unknot in $S^3-L$, spanned by a disk $\Delta\subset S^3$. Let $C_\pm$ be a pair of circles parallel to $C$ in the diagram and disjoint from $\Delta$. Connect these circles by a (possibly complicated) band $b$ in the diagram, disjoint from $\Delta$ and from the interior of the annulus $A$ bounded by $C_\pm$. (See Figure~\ref{delta}, ignoring the horizontal dashed curve.) The surface $T_0=A\cup b$ is an embedded punctured torus or Klein bottle in $S^3-L$, depending on the twisting of $b$. Let $\delta=\partial T_0$. Under the additional hypothesis that $\delta$ is unknotted in the 3-manifold $M$, we can add a suitably framed 2-handle to $I\times M$ along $\delta$ and cancel it with a 3-handle to recover $I\times M$. If $T_0$ is orientable, this 2-handle will be 0-framed in $S^3$ (since the normal to $\delta$ along $T_0$ will give the 0-framing in both $S^3$ and $M$). Otherwise, we hypothesize that its framing is 0 in $S^3$. Since $C$ is unknotted in $S^3$ and $b$ avoids $A$ and $\Delta$, we can then apply Figure~\ref{1h} to change $k$ to $k+1$ in Figure~\ref{delta}. Canceling the new twists by an isotopy in $S^3$, we return to the original diagram (ignoring the dashed curve). Assuming the 3-handle can be suitably controlled (an issue we discuss below), the net effect is a self-diffeomorphism of $M$. To see that this diffeomorphism may be nontrivial, note that it wraps the dashed curve twice around $A$ parallel to $C$.

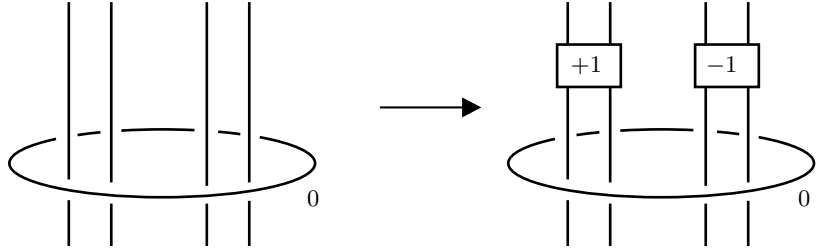
\begin{figure}
\labellist
\small\hair 2pt
\pinlabel $+1$ at 218 68
\pinlabel $-1$ at 269 68
\pinlabel $0$ at 115 18
\pinlabel $0$ at 300 18
\endlabellist
\centering
\includegraphics{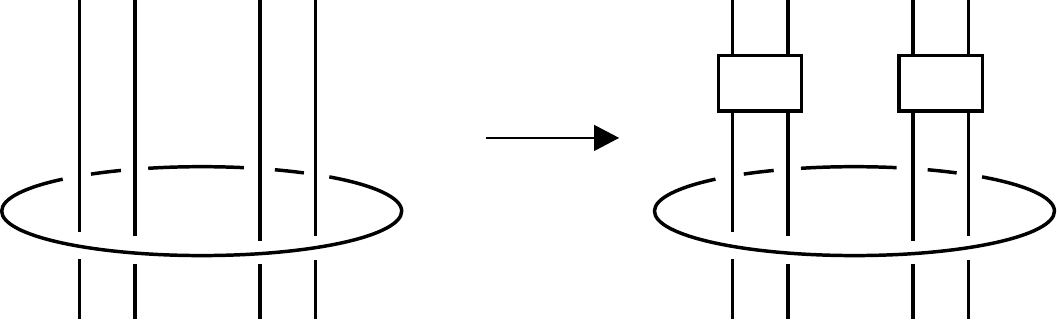}
\caption{Twisting along a 0-framed unknot.}
\label{1h}
\end{figure}

\begin{figure}
\labellist
\small\hair 2pt
\pinlabel $k$ at 56 135
\pinlabel $-k$ at 55 23
\pinlabel $C_+$ at 6 111
\pinlabel $C$ at 6 85
\pinlabel $C_-$ at 6 50
\pinlabel $\leftarrow A\subset T_0$ at 120 98
\pinlabel $b\subset T_0$ at 117 38
\pinlabel $\Delta$ at 82 86
\endlabellist
\centering
\includegraphics{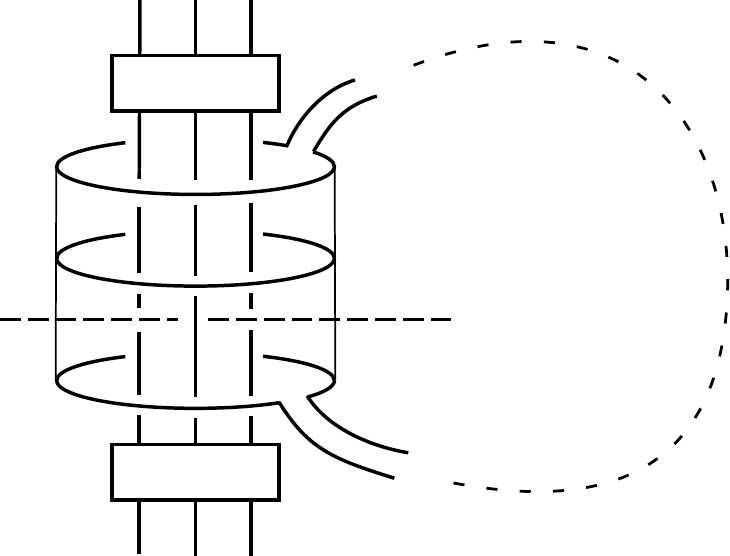}
\caption{A $\delta$-move, where $\delta=\partial T_0$ is the band-sum of $C_+$ and $C_-$ along the band $b$.}
\label{delta}
\end{figure}

\begin{de} The above diffeomorphism (when defined) is called a {\em $\delta$-move} \cite{Acork}. The corresponding link diagram of $M$, with $\delta$ unknotted in $M$ (inducing the 0-framing in $S^3$) and drawn as in Figure~\ref{delta} for an explicit choice of $b$, will be called a {\em $\delta$-move diagram}. A $\delta$-move diagram will be called {\em orientable} or {\em nonorientable} according to whether $T_0$ is orientable. It will be called {\em compressible} if there is a disk $d\subset M$ such that $d\cap T_0=\partial d$ is neither trivial nor boundary-parallel in $T_0$. It will be called {\em incompressible} otherwise.
\end{de}

The relation between surface twists and $\delta$-moves begins with the following:

\begin{prop}\label{TtoDelta} Every torus (resp.\ Klein bottle) twist on a 3-manifold $M$ is isotopic to a $\delta$-move with an orientable (resp.\ nonorientable) diagram.
\end{prop}

\begin{proof} Let $T\subset M$ be a torus or Klein bottle, containing circles $\alpha$ and $\beta$ as in Definition~\ref{Ttwist} (where $\beta$ is a section of the circle bundle in the Klein bottle case). Choose a surgery diagram of $M$ in which $\alpha$ is given by an unknot in $S^3-L$ whose 0-framing is given by the normal vectors to $\alpha$ in $T$, and whose spanning disk is disjoint from $\beta$. Then the subset $\alpha\cup\beta$ has a neighborhood in $T$ that can be identified with $T_0$ in the definition of a $\delta$-move, with $\alpha$ identified with $C$. The boundary $\delta$ of this $T_0$ explicitly bounds a disk $D\subset T\subset M$, so is unknotted in $M$ as required, and correctly framed. (Even a nonorientable $T_0$ induces the 0-framing on $\delta$ in $S^3$, as seen by using $\Delta$ to surger it to a disk.) To reinterpret the twist on $T$ parallel to $\alpha$ as a $\delta$-move, we attach a 2-handle to $\delta$ with framing 0, then cancel it with a 3-handle, whose attaching sphere can be chosen to be $D$ capped with the core of the 2-handle. The $\delta$-move is realized by blowing up a $\pm 1$-framed circle $\gamma$ at $C_+$, sliding it over the 2-handle at $\delta$ to $C_-$, blowing it back down, and canceling the twists as in Figure~\ref{delta}. If we cancel the 2-3 handle pair, the slide over the 2-handle becomes an isotopy across the disk $D\subset T\subset M$. Thus, the slide appears in $M$ as an isotopy dragging $\gamma$ from $C_+$ to $C_-$ around $T$ in the direction that avoids the intervening annulus $A$. (In $S^3$, we see a handle slide each time $\gamma$ follows $D$ over a handle.) To show that this $\delta$-move is the twist on $T$, it suffices to work in a tubular neighborhood of $T$ containing the support of the diffeomorphisms and check that the $\delta$-move diverts any curve in $M$ that crosses $T$, parallel to $\alpha$. This is true for curves intersecting $A$, as Figure~\ref{delta} shows. Other curves through $T$ will be suitably modified as in Figure~\ref{alpha} when $\gamma$ collides with them.
\end{proof}

To make progress on a converse to this proposition, we must understand the extent to which a $\delta$-move is well-defined in general. Attaching the 2-handle along $\delta$ caps off the surface $T_0$ to an embedded torus or Klein bottle $T$. However, this lives not in $M$, but rather in the manifold $M_\#=M\# S^1\times S^2$ obtained from $M$ by 0-surgery on the unknot $\delta$. If the disk $D$ in $M$ along which the 3-handle is attached is disjoint from $\inter T_0$, we can eliminate the difficulty by canceling the 2-3 pair, obtaining a torus or Klein bottle twist on $M$ as in the previous proof. However, a proposed advantage of $\delta$-moves is their apparent additional generality, so we should consider what happens when $\inter D$ is allowed to intersect $T_0$ or other surfaces in the construction. (For a specific example, start with a twist on a separating sphere, surger the sphere at its poles to an immersed Klein bottle, and interpret this as a $\delta$-move with a nonorientable diagram.) In this generality, we have a torus or Klein bottle twist $f_\#$ in $M_\#$ that we wish to interpret as a diffeomorphism of $M$. We recover $M$ from $M_\#$ by surgering out the attaching sphere $S\subset M_\#$ of the 3-handle, which is obtained by capping $D$ with the core of the new 2-handle. The first difficulty we encounter if $\inter D$ intersects $T_0$ is that $f_\#$ may move $S$. Thus, to have a well-defined diffeomorphism of $M$, we must isotope $f_\#(S)$ back to $S$ in $M_\#$ before surgering back to $M$. This is not always possible. For example, starting from a torus twist exhibited as in the previous proof, we can obtain $D$ from the obvious disk by tubing it together with an essential sphere along an arc that intersects $A$. Then $f_\#$ can change the arc by a nontrivial element of $\pi_1(M)$ so that   $f_\#$ changes the class of $S$ in $\pi_2(M)$. If we know that $f_\#$ is isotopic to a diffeomorphism preserving $S$, then it does extend over the 3-handle, so restricts to a diffeomorphism on $M$. However, this diffeomorphism need not be unique: Starting again from a torus twist, with the torus bounding a solid torus in $M$, construct a new $M'$ by connected sum with another 3-manifold. If the sum occurs outside the solid torus, the twist of $M'$ is still trivial. If it occurs inside, we can obtain a slide diffeomorphism with infinite order in $\pi_0(\Diff_+(M'))$ (detected by its effect on $\pi_2(M')$). Thus, a $\delta$-move depends in general on the particular choice of auxiliary disk $D$ capping $\delta$ in $M$. This can be difficult to specify explicitly in a diagram. To make matters worse, it is a nontrivial problem to understand the extent to which the choice of isotopy from $f_\#(S)$ back to $S$ affects the resulting diffeomorphism of $M$. Fortunately, the issue can be resolved through work of Hatcher and McCullough \cite{HM}.

\begin{thm}\label{unique} Every $\delta$-move diagram for an irreducible 3-manifold $M$ determines a unique $\delta$-move diffeomorphism up to isotopy. On a reducible manifold $M$, a $\delta$-move diagram, together with a choice of auxiliary disk $D\subset M$ spanning $\delta$ (up to isotopy rel boundary) determines at most one diffeomorphism up to isotopy and elements of order 2 in $\pi_0(\Diff_+(M))$. The latter are composites of twists on a fixed collection of disjoint spheres.
\end{thm}

\begin{proof}
Given a $\delta$-move diagram and a fixed choice of spanning disk $D\subset M$ for $\delta$, let $S\subset M_\#$ be the associated surgery sphere. Given two isotopies of $f_\#(S)$ to $S$ in $M_\#$, we wish to relate the corresponding diffeomorphisms of $M$. Before the surgery is reversed, these are related by composition with a diffeomorphism of the pair $(M_\#,S)$ that is isotopic (not preserving $S$) to the identity. By \cite[Lemma~3.4]{HM} (with $n=0$ and $S_0=S$), such a diffeomorphism, up to isotopy, comes from a composite of sphere twists on the manifold $M_1$ made by cutting $M_\#$ along $S$. We reverse the surgery by capping off the boundary components of $M_1$ with balls. If $M$ is irreducible, the spheres in question all bound balls in $M$, so their twists are isotopic to the identity. Otherwise, McCullough \cite[Section~3]{Mc} shows that the sphere twists of $M$ generate a normal subgroup $\rot(M)$ of $\pi_0(\Diff_+(M))$ isomorphic to $\oplus_r\Z_2$ for some finite $r$. We can surger $M$ on 2-spheres to get a connected sum of irreducible manifolds, and for any such presentation, the sum spheres and surgery spheres together can be assumed disjoint and comprise a generating set for $\rot(M)$. (Thus, $r$ is at most the number of prime summands of $M$, with equality only when $M=\#r S^1\times S^2$.) The reducible case of the theorem follows immediately, since any isotopy of $D$ rel boundary results in an isotopy of the corresponding diffeomorphisms of $M$. For the remaining case, suppose $M$ is irreducible. Existence follows since $S$ lies in the unique isotopy class of nonseparating spheres in $M_\#$, and uniqueness follows since the disk spanning $\delta$ in $M$ is unique up to isotopy rel boundary.
\end{proof}

Because of the difficulty of tracking isotopy classes of spanning disks in diagrams, it is natural either to assume that $M$ is irreducible or to allow the spanning disk to vary. A $\delta$-move diagram may represent more than one diffeomorphism in the reducible case (although the curve $\delta$ itself is held fixed by the diagram). We show that under broad hypotheses, every diagram represents a diffeomorphism, which can be taken to be a torus or Klein bottle twist.

\begin{thm}\label{equivalent} Every $\delta$-move diagram represents a $\delta$-move that is isotopic to a torus or Klein bottle twist parallel to $C$, provided that the 3-manifold $M$ has no $\R P^3$ summand, or that the diagram is orientable or incompressible. (The case of a Klein bottle only arises if the diagram is nonorientable.) If the diagram is compressible, the resulting twist is isotopic to the identity, provided that the diagram is orientable or $M$ is irreducible (and not $\R P^3$).
\end{thm}

\begin{cor}\label{Hsphere} Every $\delta$-move diagram for a homology sphere $M$ represents a $\delta$-move that is isotopic to a torus twist. If $M$ is also irreducible, then $\delta$-moves and torus twists comprise the same subset of $\pi_0(\Diff_+(M))$. \qed
\end{cor}
 
\begin{proof}[Proof of Theorem~\ref{equivalent}.] We begin with a $\delta$-move diagram, whose curve $\delta$ bounds an embedded disk $D\subset M$ by definition. We wish to modify $D$ so that its interior becomes disjoint from $T_0$. Then $T_0\cup D$ is an embedded torus or Klein bottle, and the proof of Proposition~\ref{TtoDelta} shows that the resulting twist is realized up to isotopy by the diagram. Recall that the surfaces $T_0$ and $D$ induce the same framing on their common boundary $\delta$, so we can assume their interiors intersect in a finite collection of circles. We can eliminate all circles bounding disks in $T_0$ by successively replacing disks in $D$ by innermost disks in $T_0$. If any innermost circle of $D$ is then boundary-parallel in $T_0$, the required new version of $D$ is obtained by joining the corresponding innermost disk to an annulus parallel to a boundary collar of $T_0$. Otherwise, either there are no remaining circles and we are done, or an innermost disk $d$ of $D$ exhibits the diagram as compressible. In the latter case, if $T_0$ is orientable, the required disk is obtained by surgering a parallel copy of $T_0$ along $d$. The resulting torus $T$ is exhibited as the boundary of a solid torus. Thus, the diagram represents a twist on the boundary of a solid torus, which is in turn isotopic to the identity. If $T_0$ is nonorientable, $\partial d$ cannot bound a M\"obius band in $T_0$, or else we could construct an embedded projective plane, whose tubular neighborhood would be an $\R P^3$ summand violating our hypotheses. Thus, $\partial d$ is the unique nonseparating circle in $T_0$ with orientable complement, namely the circle $C$ generating the $\delta$-move. Now we change tactics, modifying $T_0$: An isotopy of $M$ rotating $d$ by a full turn untwists the $\delta$-move near $C$ at the expense of adding twists on a pair of parallel copies of $d$. This isotopes the Klein bottle twist $f_\#$ in $M_\#$ to a twist on the sphere $S^*\subset M_\#$ made from $T_0$ by surgering on $d$. The disk $D$, and hence the surgery sphere $S$, can easily be made disjoint from $S^*$, so that they are not moved by the sphere twist in $M_\#$. Thus, $f_\#$ only changes $S$ by an isotopy. It follows immediately that the original diagram, together with this $S$ (or $D$) and isotopy, determines a $\delta$-move, and it is isotopic to the twist on the sphere in $M$ descending from $S^*$ by surgery on $S$. We can further surger this sphere in $M$ along a tube connecting its poles, obtaining a torus twist. Alternatively, if $M$ is irreducible, the sphere bounds a ball over which the twist extends, so the twist is isotopic to the identity.
\end{proof}

The proof also gives a more general result about the compressible case:

\begin{cor}  If $M$ has no $\R P^3$ summand, then every compressible $\delta$-move diagram represents an element of order at most 2 in $\pi_0(\Diff_+(M))$ (that is the identity in the orientable case).
\end{cor}

\begin{proof} The orientable case is given by the theorem.  Its proof shows that the nonorientable case can be reduced to a twist on a sphere.
\end{proof}

\begin{cor}  Suppose $M$ is atoroidal with no $\R P^3$ summand. Then every $\delta$-move diagram represents an element of order at most 2 in $\pi_0(\Diff_+(M))$. If the diagram is orientable or $M$ is irreducible, it represents the identity.
\end{cor}

\begin{proof} By definition, $M$ has no incompressible tori, and hence no incompressible Klein bottles. The proof of Theorem~\ref{equivalent} generates such a surface from any incompressible $\delta$-move diagram, so the previous corollary applies. The last sentence follows from the statement of the theorem.
\end{proof}

Akbulut's motivation for introducing $\delta$-moves was to study cork twisting by diagrams as in Section~\ref{Notes}, starting from a pair $Y\subset X$ and regluing $Y$. This raises a subtle technical issue. From the viewpoint of the definition of $\delta$-moves, the issue centers on the isotopy from $f_\#(S)$ to $S$ for surgery reversal. We start with a link diagram of $M=\partial Y$, and then add additional handles along a framed link $L'\subset M$ to get $X$. Given a $\delta$-move diagram for $M$, we must understand the effects of a resulting move on $L'$. To introduce the canceling 2-3 handle pair without moving $L'$, this link must be disjoint from the disk $D$ where the 3-handle attaches, a condition that can be routinely checked. The effects of the resulting torus twist $f_\#$ on $L'$ in $M_\#$ are easy to see: We can assume that $L'$ intersects the torus only in the annulus $A$ in the diagram, and then $L'$ is only changed at these intersections, by twisting parallel to $C$. However, we must then compose $f_\#$ with a diffeomorphism $g$ (isotopic to $\id_{M_\#}$) returning $S$ to its original position for surgery. Note that $g$ can be complicated; for example, any sheets of $S$ intersecting $T_0$ in parallel copies of $\beta$ will be dragged by $f_\#$ over the 2-handle attached to $\delta$. The effect of $g$ on $f_\#(L')$ is an unspecified isotopy in $M_\#$ that could cause band-summing with parallel copies of $\delta$. Reversing this isotopy could cause intersections of the link with $S$ that would prevent it from surviving the surgery. Thus, it is not clear  how the $\delta$-move affects the auxiliary link $L'$ without a more careful analysis.

An analogous problem arises from the viewpoint of expressing $\delta$-moves as torus twists. If we start from a $\delta$-move diagram, we can see where $T_0$ intersects $L'$. Given a procedure for seeing that $\delta$ is unknotted in $M$, it is routine to verify if the resulting disk $D$ is disjoint from $L'$. If $D$ is not directly visible in the diagram, however, we must assume it intersects $\inter T_0$ and apply the method of Theorem~\ref{equivalent}. This replaces $D$ by a new disk $D'$ disjoint from $\inter T_0$, and it is not generally clear whether $D'$ intersects $L'$. Since $D'$ is constructed in a neighborhood of $T_0\cup D$, it avoids every link component disjoint from $T_0$. However, the diffeomorphism is only interesting when $L'$ has nontrivial intersection with $T_0$, in which case further analysis is needed to determine whether $f$ causes unexpected movement of $L'$. This is why we exhibited $L'$ and the entire torus $T$ simultaneously in the same diagram for Proposition~\ref{fk}. As mentioned there, there are other approaches, notably drawing $T$ as an isotopy of a circle and applying Figure~\ref{alpha}.

\end{document}